\documentclass[12pt,makeidx]{amsart}
\usepackage{amsmath,amsthm,amsopn,amssymb,a4wide,mathrsfs,times,txfonts,varioref}
\usepackage[numbers,sort&compress]{natbib}
\usepackage[mathscr]{eucal}
\usepackage{enumerate,color}
\usepackage[pagebackref,colorlinks,linkcolor=red,citecolor=blue,urlcolor=blue,hypertexnames=true]{hyperref}
\usepackage{url}

\newtheorem{theorem}{Theorem}[section]
\newtheorem{lemma}[theorem]{Lemma}

\newtheorem{proposition}[theorem]{Proposition}
\newtheorem{corollary}[theorem]{Corollary}
\newtheorem*{problem*}{Problem}

\newcommand{\ip}[2]{{\left<#1,#2\right>}}
\newcommand{\zb}{{\mathcal Z(B)}}
\newcommand{\zbp}{{\mathcal Z(B')}}
\newcommand{\ran}{{\mathrm{ran}\,}}

\begin{document}

\title[Rational dilation and constrained algebras]{Rational dilation
  problems associated with constrained algebras}

\author[Dritschel]{Michael A.~Dritschel}
\address{School of Mathematics, Statistics {\&} Physics\\
  Newcastle University \\
  Newcastle upon Tyne\\
  NE1 7RU\\
  UK } \email{michael.dritschel@ncl.ac.uk}

\author[Undrakh]{Batzorig Undrakh}
\address{Institute of Mathematics \\
  National University of Mongolia \\
  Ulaanbaatar \\
  Mongolia} \email{undrakhbatzorig@gmail.com}

\thanks{This work is in part from the PhD dissertation of Batzorig
  Undrakh, under the supervision of Michael Dritschel.}


\date{\today}


\begin{abstract}
  A set $\Omega$ is a spectral set for an operator $T$ if the spectrum
  of $T$ is contained in $\Omega$, and von~Neumann's inequality holds
  for $T$ with respect to the algebra $R(\Omega)$ of rational
  functions with poles off of $\overline{\Omega}$.  It is a complete
  spectral set if for all $r\in \mathbb N$, the same is true for
  $M_r(\mathbb C)\otimes R(\Omega)$.  The rational dilation problem
  asks, if $\Omega$ is a spectral set for $T$, is it a complete
  spectral set for $T$?  There are natural multivariable versions of
  this.  There are a few cases where rational dilation is known to
  hold (eg, over the disk and bidisk), and some where it is known
  to fail, for example over the Neil parabola, a distinguished variety
  in the bidisk.  The Neil parabola is naturally associated to a
  constrained subalgebra of the disk algebra $\mathbb C + z^2
  A(\mathbb D)$.  Here it is shown that such a result is generic for a
  large class of varieties associated to constrained algebras.  This
  is accomplished in part by finding a minimal set of test functions.
  In addition, an Agler-Pick interpolation theorem is given and it is
  proved that there exist Kaijser-Varopoulos style examples of
  non-contractive unital representations where the generators are
  contractions.
\end{abstract}

\maketitle

\section{Introduction}
\label{sec:introduction}

It was first recognized in the 1950s that there is a deep connection
between the fact that over the unit disk $\mathbb D$ of the complex
plane $\mathbb C$, von~Neumann's inequality holds for any Hilbert
space contraction operator, and that a contraction can be dilated to
unitary operator (the Sz.-Nagy dilation theorem).  A similar
phenomenon is observed for a commuting pair of contractions, which
according to And\^o's theorem, dilate to a commuting pair of unitary
operators.

More generally, an operator $T$ in $\mathcal{B(H)}$, the bounded linear
operators on a Hilbert space $\mathcal H$, is said to have a
\emph{rational dilation} (with respect to a compact set
$\overline{\Omega}$) if there is a Hilbert space $\mathcal K \supset
\mathcal H$ and a normal operator $N\in \mathcal{B(K)}$ with spectrum
in the boundary of $\overline{\Omega}$ such that $f(T) = P_\mathcal H
r(N) |_\mathcal H$ for all $f\in R(\overline{\Omega})$, the rational
functions with poles off of $\overline{\Omega}$.

\smallskip

There is a natural multivariable version of this.

\begin{problem*}[Rational dilation problem\footnote{\label{fn:1}This
    problem is usually attributed to Halmos, and while this seems
    plausible, we have been unable to find a reference!}]
  \label{prob:rat_dil_prob}
  Let $\Omega$ be a domain in $\mathbb C^n$ with compact closure and
  suppose that $T$ is a commuting tuple of bounded operators on a
  Hilbert space $\mathcal H$ with spectrum contained in
  $\overline{\Omega}$.  Furthermore, assume that for every $f\in
  R(\Omega)$, the set of rational functions with poles off of
  $\overline{\Omega}$, the von~Neumann inequality holds; that is,
  $\|f(T)\| \leq \|f\|_\infty$, where $\|\cdot\|_\infty$ is the
  supremum norm over $\overline{\Omega}$.  Does there exist a Hilbert
  space $\mathcal K \supset \mathcal H$ and commuting tuple of normal
  operators $N$ on $\mathcal K$ with spectrum in the boundary of
  $\Omega$ such that $f(T) = P_\mathcal H r(N) |_\mathcal H$ for all
  $f\in R(\overline{\Omega})$?  That is, does $T$ have a
  \emph{rational dilation} to $N$?
\end{problem*}

Here Arveson~\cite{MR52:15035} is followed in defining the spectrum of
a tuple $T$ to be $\sigma(T) := \{\lambda\in\mathbb C^n: \text{ for }
p:\mathbb C^n \to \mathbb C \text{ a polynomial, } p(\lambda)\in
\sigma(p(T)) \}$.  He showed that this set is non-empty and compact,
and that the spectral mapping theorem holds for all non-constant
rational functions with poles off of $\sigma(T)$.

When the von~Neumann inequality holds for an operator (or tuple of
operators) $T$ as in the statement of the rational dilation problem,
$\overline{\Omega}$ is said to be a \emph{spectral set} for $T$.  It
is not difficult to see that if $T$ has a rational dilation, then
$\overline{\Omega}$ is a spectral set for $T$; indeed, one also has
the von~Neumann inequality for $f\in R(\Omega)\otimes M_r(\mathbb C)$,
the matrix valued rational functions with poles off of
$\overline{\Omega}$, for any finite $r$.  Hence $\overline{\Omega}$
is a \emph{complete spectral set} for $T$.

A nontrivial fact, also due to Arveson~\cite{MR52:15035}, is that $T$
has a rational dilation if and only if $\overline{\Omega}$ is a
complete spectral set for $T$.  Thus the rational dilation problem can
be reformulated as: \emph{If $\overline{\Omega}$ is a spectral set for
  $T$, is it a complete spectral set for $T$?}

Given a set $X \subset \mathbb C^d$, a function $f: X \to \mathbb C$
is \emph{analytic} if for every $x\in X$, there is an open
neighborhood of $x$ to which $f$ extends analytically.  Denote by
$A(\Omega)$ the subalgebra of functions in $C(\overline{\Omega})$
which are analytic on $\Omega$.  At least over subsets of $\mathbb C$,
there are various conditions which imply that $R(\Omega)$ is dense in
$A(\Omega)$; for example, if $\Omega$ is finitely connected, then this
is true.  In this paper we concentrate on the setting where $\Omega$
is the intersection of a variety with ${\overline{\mathbb D}}^n$.
Since the variety is the zero set of a polynomial, similar reasoning
as in the one variable setting will dictate that $R(\Omega)$ is dense
in $A(\Omega)$.  How a bounded representation acts on $R(\Omega)$ is
determined by its action on the generators, so such a representation
extends continuously to $A(\Omega)$.  This gives yet another
formulation of the rational dilation problem over suitably nice
$\Omega$: \emph{Is every contractive representation of $A(\Omega)$
  completely contractive?}

An implication of the Sz.-Nagy dilation theorem is that contractive
representations of $A(\mathbb D)$ are completely contractive, and
And\^o's theorem allows us to draw the same conclusion for $A(\mathbb
D^2)$.  So for $\Omega = \mathbb D$ or $\mathbb D^2$, rational
dilation holds.  A more substantial argument is needed to prove that
rational dilation holds for annuli~\cite{MR87a:47007} (but see
also~\cite{MR3584680}), and intriguingly, there is a way of mapping an
annulus to a distinguished variety of the bidisk~\cite{MR0241629}
(that is, a variety $\mathcal V$ which intersects $\mathbb D^2$ and
satisfies $\mathcal V \cap \partial \mathbb D^2 \subset \mathbb T^2$,
which is the \emph{distinguished}, or \emph{\v{S}ilov boundary} of
$\mathbb D^2$).  Thus rational dilation holding for annuli is
equivalent to it holding for a certain family of distinguished
varieties in $\mathbb D^2$.  It is natural to wonder if this is a
legacy of rational dilation holding over $\mathbb D^2$, and so to
speculate that perhaps rational dilation also holds for other
distinguished varieties in~$\mathbb D^2$.

Alas, this is too much to hope for.  In~\cite{MR3584680}, it was
proved that rational dilation fails for the Neil parabola $\mathcal N
= \{(z,w)\in\mathbb D^2: z^2=w^3\}$.  The techniques are indirect.  As
with an annulus, one can associate $A(\mathcal N)$ to another algebra.
In this case, there is a complete isometry mapping $A(\mathcal N)$
onto $A_{z^2}(\mathbb D) = \mathbb C + z^2 A(\mathbb D)$, the
subalgebra of $A(\mathbb D)$, the functions of which have first
derivative vanishing at $0$.  It is shown in~\cite{MR3584680} that
this algebra has a contractive representation which is not
$2$-contractive, and so not completely contractive.

In this paper, we show that rational dilation fails without fail for
algebras $A(\mathcal V_B)$ of functions which are analytic and
continuous up to the boundary on distinguished varieties $\mathcal
V_B$ of the $N$-disk associated to finite Blaschke products $B$ with
$N \geq 2$ zeros.  We also prove that it fails on associated
distinguished varieties of the $2$-disk, at least if $B$ has two or
more distinct zeros all of the same multiplicity.  This enormously
increases the set of examples where one can answer such questions.

The methods used were pioneered in~\cite{MR2946923}
and~\cite{MR3584680}, though they also have predecessors in
~\cite{MR2375060}, \cite{MR2163865} and \cite{MR2389623}.  The first
hurdle to be overcome is the construction of a minimal set of test
functions for algebras of the form $\mathcal A_B = \mathbb C +
B(z)A(\mathbb D)$, as in~\cite{MR2946923}.  Since it has $N$
generators, this algebra is completely isometrically isomorphic to
$A(\mathcal V_B)$, $\mathcal V_B$ a distinguished variety of the
$N$-disk.  It is also possible to consider the subalgebra $\mathcal
A_B^0$ of $\mathcal A_B$ generated by the first two generators, $B$
and $zB$, of $\mathcal A_B$.  This is completely isometrically
isomorphic to a subalgebra $\mathcal A(\mathcal N_B)$ on a
distinguished variety of the bidisk.  The algebras $\mathcal A_B$ were
already studied from the dual viewpoint of families of kernels
in~\cite{MR2514385}, while we present here the first systematic study
of the algebras $\mathcal A_B^0$.

For both $\mathcal A_B$ and $\mathcal A_B^0$ (with the condition on
the zeros of $B$ mentioned above), we construct an example of a
contractive representation which is not completely contractive,
yielding rational dilation results on the associated varieties.  The
strategy for doing this goes back to~\cite{MR2163865}, though was
undoubtedly familiar to Jim Agler even before this.  One shows that
there is a contractive representation which is not completely
contractive.  This is done by proving that certain matrix valued
measures arising in the so-called Agler decomposition for matrix
valued functions must diagonalize if rational dilation is to hold.
Then it is a matter of finding a function for which this does not
happen.

While it is well known that for $N>2$, $A(\mathbb D^N)$ itself has
contractive representations which are not completely contractive, it
is not \emph{a priori} the case that such a contractive representation
of $A(\mathbb D^N)$ when restricted to a subalgebra is also not
completely contractive.  As a trivial example, consider $A(\mathbb
D^2)$ in $A(\mathbb D^3)$.  Likewise, simply knowing that a function
algebra has a contractive representation which is not completely
contractive does not necessarily imply the same is true for any
algebra containing it.  The Neil algebra as a subalgebra of
$A(\mathbb D^2)$ is a case in point.

Various noteworthy observations are made in the course of the paper.
For example, for both $\mathcal A_B$ and $\mathcal A_B^0$ minimal sets
of test functions are constructed (for any $B$ with two or more
zeros), yielding optimal forms of Agler-Pick interpolation theorems.
Kaiser-Varopoulos type examples of unital representations which are
contractive on the generators of these algebras yet which fail to be
contractive representations are also found.  There is in addition a
characterization of completely contractive representations along the
line of the Sz.-Nagy dilation theorem, much like that proved by
Broschinski for the Neil algebra~\cite{MR3208801}.

The work is presented in the following order.
Section~\ref{sec:dist-vari-assoc} introduces the distinguished
varieties associated to Blaschke products on which we will study the
rational dilation problem, while Section~\ref{sec:rati-dilat-probl}
presents the rational dilation problem.
Section~\ref{sec:test-functions} outlines the notion of test functions
and their application to realization and interpolation problems.  We
show that there is no loss in generality in restricting to Blaschke
products with at least one zero at $0$.  The Herglotz representation
plays a central role, and there is a closed cone of positive measures
which is fundamental.  The extreme rays are connected with certain
probability measures which, after a Cayley transform, yield a set of
candidates for the test functions, as we see in
Section~\ref{sec:extreme-rays}.  The next, and arguably most
challenging step, is to show that the set of test functions found is
in some sense minimal.  This is addressed in
Section~\ref{sec:descr-test-funct}, and then applied in
Section~\ref{sec:compl-contr-repr} to give the Kaijser-Varopoulos
style representation mentioned above and a Sz.-Nagy type dilation
theorem.  Finally, in Section~\ref{sec:contractive-but-not-cc} we
tackle the rational dilation problem.  The paper concludes with some
remarks.

\section{Distinguished varieties associated to Blaschke products}
\label{sec:dist-vari-assoc}

We begin by describing the distinguished varieties in the bidisk
considered in this paper.

The following notation will be useful.  For $x=(x_1,\dots,x_n) \in
\mathbb C^n$, define $S_0(x) = 1$ and
\begin{equation*}
  S_k(x) = (-1)^k \sum_{1\leq x_1 \leq \cdots \leq x_n} x_{i_1}
  \cdots x_{i_n}, \quad k = 1,\dots , n,
\end{equation*}
the $k$th (signed) symmetric sum of the elements of $x$.  If $k > n$,
define $S_k(x) = 0$.  Then
\begin{equation}
  \label{eq:1}
  \prod_{j=1}^n (z-x_j) = \sum_{k=0}^n S_k(x) z^{n-k} \qquad
  \text{and} \qquad  \prod_{j=1}^n (1-\overline{x_j}z) =
  \sum_{k=0}^n S_k(\overline{x}) z^k,
\end{equation}
where $\overline{x} = ({\overline{x_1}},\dots,{\overline{x_n}})$.
Also define $S_0^{-i}(x) = -1$ and
\begin{equation*}
  S_k^{-i}(x) = S((x_1,\dots , x_{i-1},-x_i,x_{i+1},\dots, x_n)),
  \qquad k=1,\dots , n.
\end{equation*}
Then $S_n^{-i}(x) = -S_n(x)$.  For $x=\lambda \subset \mathbb T^n$,
\begin{equation*}
  S_k(\lambda) = S_n(\lambda)S_{n-k}(\overline{\lambda})
  \qquad\text{and}\qquad
  S_k^{-i}(\lambda) =
  -S_n^{-i}(\lambda)S_{n-k}^{-i}(\overline{\lambda}).
\end{equation*}

Let
\begin{equation*}
  B(z) = \prod_{k=1}^N \frac{z-\alpha_k}{1-\overline{\alpha_k}z},
  \qquad z\in \overline{\mathbb D},
\end{equation*}
be a Blaschke product with at least two not necessarily distinct
zeros.  Define $x = B(z)$, $y = zB(z)$.  Then the pair $(x,y) \in
\overline{\mathbb D}^2$.  Since
\begin{equation}
  \label{eq:2}
  B(z)^{N+1}\prod_{k=1}^N (1-\overline{\alpha_k}z) =
  B(z)^N\prod_{k=1}^N (z-\alpha_k),
\end{equation}
the pair $(x,y)$ satisfies the polynomial identity $P(x,y) = 0$, with
\begin{equation*}
  \begin{split}
    P(x,y) & = x\prod_{k=1}^N (x-\overline{\alpha_k} y)
    - \prod_{k=1}^N (y-\alpha_k x) \\
    & = \sum_{k=0}^N \left(S_k(\overline{\alpha}) x^{N-k+1} y^k -
      S_k(\alpha) x^k y^{N-k}\right)
  \end{split}
\end{equation*}

Now suppose $(x,y)\in\mathbb D^2$ is any point satisfying $P(x,y) =
0$.  If $x=0$, then $y^N = 0$, and thus $y = 0$.  So assume $x\neq 0$.
Letting $z = y/x$, it follows that $x = B(z)$ and $y = zB(z)$, and so
$x$ and $y$ have the form indicated above.  Furthermore, since $|B(z)|
\leq 1$ if and only if $|z| \leq 1$, $(x,y)\in \overline{\mathbb D}^2$
if and only if $z\in \overline{\mathbb D}$.

The locus described by $P(x,y) = 0$ defines a variety in $\mathbb
C^2$.  Furthermore, since $|x| = |B(z)| = 1$ implies that $|z| = 1$,
$|x| = 1$ implies $|y| = 1$.  Likewise, if $|y| = 1$, then the modulus
of the Blaschke product $zB(z)$ is $1$, and so once again $|z| = 1$,
implying that $|x| = 1$.  Hence this variety intersects the boundary
of $\overline{\mathbb D}^2$ in $\mathbb T^2$, which is the \v{S}ilov
or \emph{distinguished boundary}; that is, $P(x,y) = 0$ defines a
\emph{distinguished variety}.  Write $\mathcal N_B$ for $\{(x,y) \in
\overline{\mathbb D}^2 : P(x,y) = 0\}$.

The variety $\mathcal N_B$ has a singularity solely at $(0,0)$.  If
all of the zeros of $B$ are distinct, then there is an $(N-1)$-fold
crossing at this point.  At the other extreme, if all the zeros are
the same, there is a cusp (for example, this is what happens with the
Neil parabola, where $B(z) = z^2$).  Intermediate cases give rise to
mixtures of these.

The general theory of distinguished varieties of the bidisk as laid
out by Agler and McCarthy in \cite{MR2231339} (see also,
\cite{MR2661491}) shows that such varieties have a \emph{determinantal
  representation}.

\begin{theorem}[\cite{MR2231339}]
  \label{thm:det-repn-for-dist-vars}
  Let $\mathcal V$ be a distinguished variety, defined as zero set of
  a polynomial $p \in \mathbb{C}[x,y]$ of minimal degree $(m, n)$.
  Then, there is an $(m+n)\times (m+n)$ unitary matrix $U$, written as
  $U= \begin{pmatrix} A & B\\ C & D \end{pmatrix}:\mathbb{C}^{m}
  \oplus \mathbb{C}^n \to \mathbb{C}^{m} \oplus \mathbb{C}^n$, such
  that
  \begin{enumerate}
  \item $A$ has no unimodular eigenvalues,
  \item $p(x,y)$ is a constant multiple of
    \begin{equation*}
      \det \begin{pmatrix} D - x I_n & x C \\ B & y A - I_m
      \end{pmatrix}, \qquad\text{and}
    \end{equation*}
  \item for the rational matrix valued inner function $\Psi(y) = D + y
    C \left(I_m - y A \right)^{-1}B$,
    \begin{equation*}
      \mathcal V = \left\{ (x,y) \in \overline{\mathbb{D}}^2 :
        \det(xI_n - \Psi(y))=0 \right\}.
    \end{equation*}
  \end{enumerate}
  Moreover, if\, $\Psi$ is a matrix valued rational inner function on
  $\overline{\mathbb{D}}$, then
  \begin{equation*}
    \left\{ (x, y) \in \overline{\mathbb{D}}^2 : \det(xI_n -
      \Psi(y))=0 \right\}
  \end{equation*}
  is a distinguished variety.
\end{theorem}

A straightforward calculation shows that a determinantal
representation for $\mathcal N_B$ is obtained by choosing
\begin{equation*}
  \Psi(y) = \begin{pmatrix}
    0 & -y & 0 & \cdots & 0 & 0\\
    0 & \overline{\alpha_1}y & -y & 0 & \cdots & 0\\
    0 & 0 & \overline{\alpha_2}y & -y & \ddots & \vdots\\
    \vdots & \ddots & \ddots & \ddots & \ddots& \vdots\\
    0 & \ddots & \ddots & \ddots & \overline{\alpha_{N-1}}y & -y \\
    (-1)^N & 0 & \cdots & \cdots & 0 & \overline{\alpha_N}y
  \end{pmatrix}
  T^{-1},
\end{equation*}
where
\begin{equation*}
  T =
  \begin{pmatrix}
    1 &  -\alpha_1 & 0 & \cdots & \cdots & 0\\
    0 & 1 &  -\alpha_2 & 0 & \cdots & \vdots \\
    \vdots & \ddots & \ddots & \ddots & \ddots & \vdots \\
    \vdots & \ddots & \ddots & \ddots & \ddots & 0 \\
    \vdots & \ddots & \ddots & 0 & 1 &  -\alpha_n \\
    0 & \cdots & \cdots & \cdots & 0 & 1
  \end{pmatrix}.
\end{equation*}

By defining variables $x_j = z^{j-1}B$, $j = 1, \dots , N$, and again
using \eqref{eq:2}, it is not hard to work out that there is an
associated distinguished variety $\mathcal V_B$ in ${\overline{\mathbb
    D}}^N$.  There will in general be multiple varieties which can be
described with these variables from \eqref{eq:2}.  They are obtained
one from the other via the identities $x_jx_{N-j}=x_ix_{N-i}$.

Obviously, when $N$ is sufficiently large, intermediate cases could be
considered, associated to algebras on distinguished varieties in
${\overline{\mathbb D}}^n$, $2< n < N$.  The techniques needed to
handle these are identical to those presented for $\mathcal A_B^0$ and
$\mathcal A_B$.

Recall the notation $A(\mathcal V_B)$ for the algebra of analytic
functions on $\mathcal V_B \cap{\overline{D}}^N$ which extend
continuously to the boundary with the supremum norm, and $\mathcal A_B
= \mathbb C + B\,A(\mathbb D)$ the associated subalgebra of $A(\mathbb
D)$.  If $B(z) = \prod_1^N\frac{z-\alpha_j}{1 -
  \overline{\alpha_j}z}$, then $\mathcal A_B = \mathbb C + \prod_1^N
(z-\alpha_j)\,A(\mathbb D)$ since $\prod_1^N (1 -
\overline{\alpha_j}z) \in A(\mathbb D)$.  Thus $\mathcal A_B$ is
generated by $\{Bz^j\}_{j=0}^{N-1}$.  In connection with $A(\mathcal
N_B)$, we are also interested in a subalgebra of $\mathcal A_B$
generated by $B$ and $zB$, denoted by $\mathcal A_B^0$.

\begin{theorem}
  \label{thm:algebras-isom-isom}
  The algebra $A(\mathcal V_B)$ is completely isometrically isomorphic
  to the algebra $\mathcal A_B$.  The algebra $A(\mathcal N_B)$ is
  completely isometrically isomorphic to the algebra $\mathcal A_B^0$,
  which consists of those functions in $\mathcal A_B$ for which the
  coefficients of terms of the form $z^i B^j$, $j = 0,\dots, N-2$ and
  $i = j+1, \dots N-1$, are $0$.  The algebra $\mathcal A_B^0$ is of
  codimension $N(N-1)/2$ in $A(\mathbb D)$ and contains $\mathbb C +
  B^{N-1} A(\mathbb D)$; so in particular, when $N = 2$, $\mathcal
  A_B^0 = \mathcal A_B$.
\end{theorem}

\begin{proof}
  The case of $A(\mathcal N_B)$ is only treated, the other being
  handled identically.

  As noted earlier, the algebra $\mathcal Q_B$ of rational functions
  with poles off of $\mathcal N_B$ is dense in $A(\mathcal N_B)$.
  Define a map $\rho:\mathcal Q_B \to \mathcal A_B$ by
  \begin{equation*}
    \rho(p/q) = \frac{p(B,zB)}{q(B,zB)}, \qquad p,q \text{
      polynomials},
  \end{equation*}
  and extending linearly.  If it were the case that $q(B(\zeta),\zeta
  B(\zeta)) = 0$ for some $\zeta\in \overline{\mathbb D}$, then for
  $(x,y) = (B(\zeta),\zeta B(\zeta)) \in \mathcal N_B$, $q(x,y) = 0$,
  and so $p/q$ cannot be in $\mathcal A(\mathcal N_B)$.  Hence the
  image of $\rho$ is in $A(\mathbb D)$.  Since the image is generated
  by $B$ and $zB$, it equals $\mathcal A_B^0$.  For $f\in A(\mathcal
  N_B)$, the maximum modulus principle holds for $\rho(f) = f(B,zB)$.
  Since $(x,y) \in \mathcal N_B \cap \mathbb T^2$ if and only if the
  associated $z$ is in $\mathbb T$, $f$ achieves its maximum modulus
  on $(x,y)\in \mathbb T^2\cap \mathcal N_B$.  Hence the map is
  isometric.  The same reasoning shows that the map is a complete
  isometry, and so it extends to a completely isometric homomorphism
  from $A(\mathcal N_B)$ onto $\mathcal A_B^0$.

  Now turn to the description of $\mathcal A_B^0$ in $\mathcal A_B$.
  Suppose for the time being that $B$ has three or more zeros, and
  that there is some $f\in A(\mathcal N_B)$ such that $\rho(f) = z^2B
  \in \mathcal A_B^0$.  Then $\rho(xf) = z^2B^2 = \rho(y^2)$.  Since
  the map $\rho$ is isometric, this implies that $xf = y^2$ in an open
  neighborhood $U$ of $(0,0)$.  Fix any non-zero complex number $t$
  and let $C_t = \{(x,y)\in \mathbb C^2 : x = ty^2\}$.  For $y_0$
  small enough and non-zero, $(x_0,y_0) \neq (0,0)$ is in $C_t \cap
  U$.  Evaluating at $(x_0,y_0)$ gives $f(x_0,y_0) = 1/t$.  Hence $f$
  cannot be analytic, and so $z^2B$ is not in $\mathcal A_B^0$.

  The same argument shows that any term of the form $z^i B^j$, $j =
  1,\dots, N-2$ and $i = j+1, \dots N-1$, while in $\mathcal A_B$, is
  not in $\mathcal A_B^0$.  Obviously anything of this form where $j$
  is arbitrary and $i\leq j$ can be written as a product of powers of
  $B$ and $zB$.

  Now suppose $B$ has $N\geq 2$ zeros and let $j = N-1$.  Then
  \begin{equation*}
    z^N B^{N-1} = \left(\prod_{j=1}^N (1-\overline{\alpha_j}z)B -
      g\right) B^{N-1} 
    = \left(\sum_{j=0}^N S_j(\overline{\alpha})z^j B - g\right)
    B^{N-1},
  \end{equation*}
  where $\deg g \leq N-1$.  All terms have the form $c z^i B^j$, $c$ a
  constant and $i \leq j$, and hence are in $\mathcal A_B^0$.  Also,
  \begin{equation*}
    z^{N+k} B^{N-1} = z^k \left(\sum_0^N S_j(\overline{\alpha})z^j B -
      g\right) B^{N-1},
  \end{equation*}
  so by an induction argument, all of these are in $\mathcal A_B^0$ as
  well.  Hence, $\mathcal A_B^0 \supset B^{N-1} A(\mathbb D)$.  In
  particular, if $B$ has only two zeros, $B$ and $zB$ generate the
  algebra $\mathcal A_B$, and in this case $\rho$ is onto.
\end{proof}

\section{The rational dilation problem and constrained algebras}
\label{sec:rati-dilat-probl}

Our goal is to study the \emph{rational dilation problem} on $\mathcal
V = \mathcal V_B$ (respectively, $\mathcal N_B$).  Thus we consider
$n$ tuples of commuting operators $T = (T_1,\dots,T_n)$ ($n=N$ and
$n=2$, respectively) acting on a Hilbert space $\mathcal H$ having
$\mathcal V$ as a \emph{spectral set}.  Recall that this means that
the joint spectrum of $T$ lies in $\mathcal V$ and for each $f\in
\mathcal Q_B$, $\|f(T)\| \leq \|f\|$, where the left hand norm is the
usual operator norm, and the right hand norm is the supremum norm of
$f$ on $\mathcal V$.  This is a form of the von~Neumann inequality,
and as noted in the introduction, can be interpreted as stating that
$T$ induces a contractive unital representation of $\mathcal Q_B$, and
hence $A(\mathcal V)$.

The \emph{rational dilation problem} then asks whether such a $T$
dilates to a commuting tuple of normal operator $W = (W_1,\dots,W_n)$
acting on some Hilbert space $\mathcal K \supset \mathcal H$ with
spectrum contained in the distinguished boundary of $\mathcal V
\subset \mathbb T^n$.  By a \emph{dilation}, we mean that $f(T) =
P_\mathcal H f(W) |_\mathcal H$ for all $f\in \mathcal Q_B$.  The
tuple $W$ is referred to as a \emph{normal boundary dilation}.  If $W$
exists for all such $T$, rational dilation is said to hold, and
otherwise, it fails.

For the tuple of normal operators $W$, not only is it the case that
$\|f(W)\| \leq \|f\|$ for $f\in \mathcal Q_B$, but also for $f\in
\mathcal Q_B\otimes M_r(\mathbb C)$, $r\in \mathbb N$.  Therefore if
$T$ has a normal boundary dilation, it is also true that $\|f(T)\|
\leq \|f\|$ for $f\in \mathcal Q_B\otimes M_r(\mathbb C)$.  In other
words, when rational dilation holds, contractive representations of
$\mathcal Q_B$ (and hence $A(\mathcal V)$) are completely contractive,
and the converse also holds.  Thus a strategy for showing that
rational dilation fails on $A(\mathcal V_B)$ (respectively,
$A(\mathcal N_B)$) is to find a contractive representation of
$\mathcal A_B$ (respectively, $\mathcal A_B^0$) which is not
completely contractive.  This is the approach taken.

\section{Test functions}
\label{sec:test-functions}

Our method for solving the rational dilation problem requires finding
a family of so-called ``test functions'' for the algebras $\mathcal
A_B$and $\mathcal A_B^0$.  For other purposes (such as solving
interpolation problems), it is useful for this family to be in some
sense minimal.  We give a brief synopsis the notion of test functions
and their use in the solution of interpolation problems, and otherwise
refer to \cite{MR2389623} for further details.  See also
\cite{MR2003b:47001}.

Let $X$ be a set and $\Psi = \{\psi_\alpha\}$ a collection of complex
valued functions on $X$.  The elements of $\Psi$ are called \emph{test
  functions} if they satisfy two conditions:
\begin{itemize}
\item For any $x\in X$, $\sup_{\psi\in \Psi} |\psi(x)| < 1$, and
\item The elements of $\Psi$ separate the points of $X$.
\end{itemize}

Given a set of test functions $\Psi$, the set of \emph{admissible
  kernels} $\mathcal K_\Psi$ consists of positive kernels $k$ on
$X\times X$ to $\mathbb C$ with the property that for each
$\psi\in\Psi$, the kernel
\begin{equation*}
  \left((1 - \psi(x)\psi(y)^*)k(x,y)\right) \geq 0.
\end{equation*}

The admissible kernels allow us to define a function algebra
$H^\infty(\mathcal K_\Psi)$ of those functions $\varphi$ on $X$ for
which there is a $c\in \mathbb R^+$ such that for all $k\in \mathcal
K_\Psi$,
\begin{equation*}
  \left((c1 - \varphi(x)\varphi(y)^*)k(x,y)\right) \geq 0.
\end{equation*}
The infimum over all such $c$ defines a norm on $H^\infty(\mathcal
K_\Psi)$ making it a Banach algebra.  Obviously, the test functions
are in the unit ball of this algebra.  Because any positive kernel
which is zero when $y\neq x$ is admissible, the norm of
$H^\infty(\mathcal K_\Psi)$ will always be greater than or equal to
the supremum norm, and so $H^\infty(\mathcal K_\Psi)$ is weakly closed
(that is, closed under pointwise convergence).

A key result in the study of algebras generated through test functions
is the realization theorem~\cite{MR2389623}, which gives several
equivalent characterizations of membership of the closed unit ball of
the algebra $H^\infty(\mathcal K_\Psi)$.  The relevant portion is
stated here.  Some notation: $C(\Psi)$ is the algebra of bounded
continuous functions on $\Psi$, and $C(\Psi)^*$ is its continuous
dual.  Assume that $\Psi$ is endowed with a suitable topology so that
for all $x\in X$, the functions $E_x : \psi \in \Psi \mapsto \psi(x)$
are in $C(\Psi)$.  In this case $E_x^* : \psi \in \Psi \mapsto
\psi(x)^*$ is also in $C(\Psi)$.

\begin{theorem}[Realization theorem]
  \label{thm:realization-theorem}
  Let $\Psi$ be a collection of test functions on a set $X$, $\mathcal
  K_\Psi$ the admissible kernels, and $H^\infty(\mathcal K_\Psi)$ the
  associated function algebra.  For $\varphi: X\to \mathbb C$, the
  following are equivalent:
  \begin{enumerate}
  \item $\varphi \in H^\infty(\mathcal K_\Psi)$ with $\|\varphi\| \leq
    1$;
  \item There is a positive kernel $\Gamma : X\times X \to C(\Psi)^*$
    such that for all $x,y\in X$,
    \begin{equation*}
      1 - \varphi(x)\varphi(y)^* = \Gamma(x,y) (1 - E_x E_y^*); \quad
      and
    \end{equation*}
  \item If $\pi$ is any unital representation of $H^\infty(\mathcal
    K_\Psi)$ mapping the elements of $\Psi$ to strict contractions
    (ie, norm strictly less than $1$), then $\pi$ is contractive.
  \end{enumerate}
\end{theorem}

The proof of the realization theorem is the basis for the following
interpolation theorem.

\begin{theorem}[Agler-Pick interpolation theorem]
  \label{thm:agler-pick-interpolation}
  Let $\Psi$ be a collection of test functions on a set $X$, $\mathcal
  K_\Psi$ the admissible kernels, and $H^\infty(\mathcal K_\Psi)$ the
  associated function algebra.  Fix a finite set $F \subset X$.  For
  $f: F\to \mathbb C$, the following are equivalent:
  \begin{enumerate}
  \item There is a function $\varphi \in H^\infty(\mathcal K_\Psi)$
    with $\|\varphi\| \leq 1$ such that $\varphi|_F = f$, and
  \item there is a positive kernel $\Gamma : F\times F \to C(\Psi)^*$
    such that for $x,y\in F$,
    \begin{equation*}
      1 - \varphi(x)\varphi(y)^* = \Gamma(x,y) (1 - E_x E_y^*).
    \end{equation*}
  \end{enumerate}
\end{theorem}

In summary, given a collection of test functions $\Psi$, first
construct a set of admissible kernels $\mathcal K_\Psi$, and then from
this a function algebra $H^\infty(\mathcal K_\Psi)$.  In most
situations though, an algebra $\mathcal A$ on a domain $X$ is already
at hand, and so for example, if one wishes to solve interpolation
problems in $\mathcal A$, it is necessary to find a set of test
functions $\Psi$ generating $\mathcal A$.  A trivial choice
(disregarding possible degeneracies) is to let $\Psi$ be the unit ball
of $\mathcal A$.  The ideal though is to choose $\Psi$ to be as small
as possible.  Care is needed since it may be the case that removing
finitely, or even countably many test functions still leaves a
suitable set of test functions.  Insisting that the set of test
functions be (weakly) compact avoids this difficulty.  Even then, the
minimal set of test functions will only be defined up to automorphism.
In any case, a compact family of test functions $\Psi$ is said to be
\emph{minimal} for an algebra $\mathcal A(\mathcal K_\Psi)$ if there
is no proper closed subset of $\Psi$ such that the realization theorem
holds for all functions in the unit ball of $\mathcal A(\mathcal
K_\Psi)$.

Let us return our attention to the constrained algebras $\mathcal A_B$
and $\mathcal A_B^0$, and the construction of minimal sets of test
functions $\Psi_B$ and $\Psi_B^0$ for these, or rather, for the weak
closure of these algebras, $H^\infty_B$ and $H^{\infty,0}_B$.  To
simplify the work, it may be assumed that one of the zeros of $B$,
written as $\alpha_0$, equals $0$.  It turns out that this assumption
in fact imposes no real restriction.

For suppose that $B$ is a Blaschke product with zeros $\mathcal Z(B) =
\{\alpha_0, \dots, \alpha_n\}$ such that no $\alpha_j = 0$.  Composing
$B$ with the M\"obius map $m_{-\alpha_0} =
(z+\alpha_0)/(1+\overline{\alpha_0}z)$, to get a Blaschke product $B'$
with zeros ${\{{\alpha'}_j = m_{\alpha_0}(\alpha_j)\}}_{j=0}^n$.
Hence ${\alpha'}_0 = 0$.  Obviously composing with $m_{\alpha_0}$ maps
$B'$ back to $B$.  Since composition with $m_{\pm\alpha_0}$ leaves
$H^\infty(\mathbb D)$ invariant, $f\in H^\infty_{B}$ if and only if
$f' = f\circ m_{-\alpha_0} \in H^\infty_{B'}$, and furthermore, $\|f\|
= \|f'\|$.

Let $\Psi_{B'}$ be a family of test functions for $H^\infty_{B'}$, and
define $\Psi = \{\psi'\circ m_{\alpha_0} : \psi' \in \Psi_{B'}\}$.
Since $m_{\alpha_0}$ is an automorphism of the disk, $\Psi_{B'}$ maps
injectively onto $\Psi$, and so it is possible to identify
$C(\Psi_{B'})$ and $C(\Psi)$.  For $x\in \mathbb D$, set $x' =
m_{\alpha_0}(x)$.  Then
\begin{equation*}
  E_x(\psi) = \psi(x) = \psi'(m_{\alpha_0}(x)) = \psi'(x') =
  E_{x'}(\psi').
\end{equation*}

Let $\varphi \in H^\infty_{B}$ and set $\varphi' = \varphi \circ
m_{-\alpha_0}$.  Assume $\|\varphi'\|\,(= \|\varphi\|) = 1$.  By the
realization theorem and the assumption that $\Psi_{B'}$ is a family of
test functions for $H^\infty_{B'}$, there is a positive kernel
$\Gamma':\mathbb D \times \mathbb D \to C(\Psi_{B'})^*$ such that for
$x,y\in\mathbb D$, and $x' = m_{\alpha_0}(x)$, $y' = m_{\alpha_0}(y)$,
\begin{equation*}
  1 - \varphi(x) \varphi(y)^* = 1 - \varphi'(x') \varphi'(y')^*
  = \Gamma'(x',y') (1 - E_{x'} E_{y'}^*)
  = \Gamma(x,y)(1 - E_x E_y^*),
\end{equation*}
where $\Gamma(x,y) = \Gamma'(m_{\alpha_0}(x),m_{\alpha_0}(y))$ is a
positive kernel from $\mathbb D \times \mathbb D$ to $C(\Psi)^*$.
Conclude that $H^\infty_B$ is in the algebra $\mathcal A$ induced by
the test functions $\Psi$ and $\varphi$ is in the unit ball of
$\mathcal A$.  Since the norm of $\varphi$ in $\mathcal A$ is greater
than or equal to the supremum norm (the norm in $H^\infty_B$), the two
norms must be equal.  On the other hand, if $\varphi$ is in the unit
ball of $\mathcal A$, then for $\varphi' = \varphi \circ
m_{-\alpha_0}$, the realization theorem implies that $\varphi' \in
H^\infty_{B'}$, and so $\varphi \in H^\infty_B$.  Thus $\mathcal A =
H^\infty_B$ with the same norm, and the conclusion is that $\Psi$ is a
family of test functions for $H^\infty_B$.  By similar arguments,
$\Psi$ is minimal if and only if $\Psi_{B'}$ is minimal.

The same argument works when dealing with $\mathcal A_B^0$, so as
needed, $B$ will be replaced by $B'$, where $0 \in \mathcal Z(B')$.

\section{Herglotz representations and extreme rays}
\label{sec:extreme-rays}

In this section, sets of test functions for the algebras $\mathcal
A_B$ and $\mathcal A_B^0$ are determined.  The strategy employed is as
follows.  Suppose that $\varphi$ is in the unit ball of one of these
algebras and that $\varphi(0) = 0$.  A Cayley transform uniquely
associates to this a function $f: \mathbb D \to \mathbb H$ with $f(0)
= 1$, where $\mathbb H$ is the right half plane in $\mathbb C$.  The
function $f$ has a Herglotz representation with respect to a unique
probability measure $\mu$ on $\mathbb T$.  The constraints of the
algebra are encoded in the measure.  The probability measures form a
compact convex set, and $\mu$ can be represented as the integral with
respect to a measure supported on the extreme points of this set.  The
set of inverse Cayley transforms of the functions which (modulo a
unimodular constant) have Herglotz representations with respect to the
extremal measures is then the candidate for the set of test functions.

If $\varphi\in H^\infty$ with $\varphi(0) = 0$ and norm at most $1$,
and $f = M\circ \varphi$ where $M(z) = \frac{1+z}{1-z}$ maps $\mathbb
D$ to $\mathbb H$, then $\mathrm{Re}\, f \geq 0$ and $f(0) = 1$.  The
map $M$ has inverse $M^{-1}(z) = \frac{1-z}{1+z}$, and hence there is
a one to one correspondence between the set of functions in the unit
ball of $H^\infty$ which are zero at $0$ and the set of holomorphic
functions mapping $\mathbb D$ to $\mathbb H$ and the value $1$ at $0$.

By the Herglotz representation theorem, for any holomorphic $f:
\mathbb D \to \mathbb H$ with $f(0) = 1$, there is a unique
probability measure $\mu$ on $\mathbb T$ (usually referred to as the
\emph{Clark} or \emph{Alexandrov-Clark measure}) such that
\begin{equation*}
  f(z) = \int_\mathbb T \frac{w+z}{w-z}\, d\mu(w),
\end{equation*}
and conversely, if $\mu$ is a probability measure on $\mathbb T$, then
\begin{equation*}
  f(z) := \int_{\mathbb T} \frac{w + z}{w - z}\,d\mu(w)
\end{equation*}
defines a holomorphic function on $\mathbb D$ to $\mathbb H$ with
$f(0) = 1$.

The following can be cobbled together from other sources (see, for
example, \citep[Chapter~9]{MR2215991}).  We give a direct and
elementary proof.

\begin{lemma}
  \label{lem:finite-meas-gives-B-prod}
  Let $\mu$ be a positive finite atomic measure on $\mathbb T$, $\mu =
  \{(\lambda_j, m_j)\}_{j=1}^n \subset \mathbb T \times \mathbb
  R_{>0}$, with $f$ the function having Herglotz representation with
  this measure.  Then $\varphi = M^{-1}\circ f$ is a unimodular
  constant multiple of a Blaschke product with $n$ zeros, counting
  multiplicities, and $\varphi(0)\in \mathbb R$.

  Conversely, given a Blaschke product $\varphi$ with $n$ zeros
  $\{\alpha_j\}$ counting multiplicities such that $\varphi(0)\in
  \mathbb R$, there is a positive finite atomic measure $\mu$ on
  $\mathbb T$ such that $f = M\circ \varphi(z)$ has a Herglotz
  representation with this measure.  Furthermore, $\mu$ is a
  probability measure if and only if $\varphi(0) = 0$.
\end{lemma}

\begin{proof}
  Let
  \begin{equation*}
    f(z) = \int_{\mathbb T} \frac{w + z}{w - z}\,d\mu(w) =
    -\sum_{i=1}^n m_i \frac{z+\lambda_i}{z-\lambda_i},
  \end{equation*}
  a holomorphic function from $\mathbb D$ to $\mathbb H$.  Set $m =
  \sum m_i$.  Then
  \begin{equation*}
    \begin{split}
      1 \pm f(z) &= \frac{\tfrac{1}{m} \sum_i m_i \prod_j(z-\lambda_j)
        \mp \sum_i m_i (z+\lambda_i)\prod_{j\neq
          i}(z-\lambda_j)}{\prod_j(z-\lambda_j)} \\
      &=\frac{\sum_{k=0}^n \left[\sum_{i=1}^n (\tfrac{m_i}{m}
          S_k(\lambda) \mp m_i S_k^{-i}(\lambda))\right]
        z^{n-k}}{\prod_j(z-\lambda_j)},
    \end{split}
  \end{equation*}
  and
  \begin{equation*}
    \varphi(z) := (M^{-1}\circ f)(z) = \frac{1-f(z)}{1+f(z)}
    = \frac{\sum_{k=0}^n \left[\sum_{i=1}^n (\tfrac{m_i}{m} S_k(\lambda)
        + m_i S_k^{-i}(\lambda))\right]z^{n-k}}{\sum_{k=0}^n \left[\sum_{i=1}^n
        (\tfrac{m_i}{m} S_k(\lambda) - m_i S_k^{-i}(\lambda))\right]z^{n-k}}
  \end{equation*}
  is a holomorphic map of the disk to itself.

  Since the coefficient of $z^n$ in the numerator of $\varphi$ is $1+m
  >0$, the numerator is a polynomial of degree $n$ with complex roots
  $\alpha_1,\dots ,\alpha_n$.  Express the numerator as
  $(1+m)\prod_j(z-\alpha_j)$.  Then
  \begin{equation*}
    (1+m)S_k(\alpha) = \sum_i (m_i S_k(\lambda) + m_i
    S_k^{-i}(\lambda)),
  \end{equation*}
  and so the denominator can be expressed as
  \begin{equation*}
    \begin{split}
      \sum_{k=0}^n \left[\sum_{i=1}^n (\tfrac{m_i}{m} S_k(\lambda) -
        m_i S_k^{-i}(\lambda))\right] z^{n-k} =\,& S_n(\lambda)
      \sum_{k=0}^n \left[\sum_{i=1}^n (\tfrac{m_i}{m}
        \overline{S_{n-k}(\lambda)} + m_i
        \overline{S_{n-k}^{-i}(\lambda)})\right] z^{n-k} \\
      = S_n(\lambda) (1+m) \sum_{k=0}^n \overline{S_k(\alpha)}z^k =\,&
      S_n(\lambda) (1+m) \prod_{j=1}^n (1-\overline{\alpha_j}z).
    \end{split}
  \end{equation*}
  Hence
  \begin{equation*}
    \varphi(z) = \overline{S_n(\lambda)} \prod_{j=1}^n
    \frac{z-\alpha_j}{1-\overline{\alpha_j}z}.
  \end{equation*}
  Since $f(0) = \sum_i m_i \in \mathbb R$, the same is then true for
  $\varphi(0)$.

  Conversely, assume that $\varphi = c B$, where $c$ is a unimodular
  constant, $B$ is a Blaschke product with $n$ zeros $\alpha_1,\dots ,
  \alpha_n$, counting multiplicities and $\varphi(0) \in \mathbb R$.
  Then
  \begin{equation}
    \label{eq:3}
    f(z) = \frac{1+\varphi(z)}{1-\varphi(z)} = \frac{\prod_j
      (1-\overline{\alpha_j}z) + c\prod_j(z-\alpha_j)}{\prod_j
      (1-\overline{\alpha_j}z) - c\prod_j(z-\alpha_j)}
  \end{equation}
  is a holomorphic map from $\mathbb D$ to $\mathbb H$.  Since
  $|S_n(\alpha)| < |c| = 1$, the leading coefficient in the
  denominator $C = \overline{S_n(\alpha)} - c$ is non-zero.  Thus the
  denominator of $f$ has $n$ zeros in $\mathbb C \backslash \mathbb
  D$, $\lambda_1,\dots, \lambda_n$.  Write the denominator as
  $C\prod_j (z-\lambda_j)$.

  If the numerator and denominator of $f$ have a common root $w$, then
  $\prod_j(w-\alpha_j) = 0$, implying $\lambda_k = \alpha_j\in\mathbb
  D$ for some $k$ and $j$, which is a contradiction.  The constant
  coefficient of the denominator equals $(1 - c S_n(\alpha))/C =
  c\overline{C}/C$, which has absolute value $1$.  Hence each
  $\lambda_j \in \mathbb T$.

  Suppose that the denominator of $f$ has a repeated root at some
  $\lambda\in \mathbb T$.  Then the logarithmic derivative of
  $\varphi$,
  \begin{equation*}
    \frac{\varphi'(z)}{\varphi(z)} = \sum_{k=1}^n
    \frac{1-|\alpha_k|^2}{(1-\overline{\alpha_k}z)(z-\alpha_k)} =
    \frac{-2f'(z)}{1-f(z)^2},
  \end{equation*}
  is zero at $\lambda$.  On the other hand, $\lambda\in \mathbb T$
  implies
  \begin{equation}
    \label{eq:4}
    \frac{\varphi'(\lambda)}{\overline{\lambda}\varphi(\lambda)} =
    \sum_k\frac{1-|\alpha_k|^2}{|\lambda - \alpha_k|^2} >0,
  \end{equation}
  giving a contradiction.

  Consequently, since the denominator of $f$ has $n$ simple roots, $f$
  has a partial fraction decomposition
  \begin{equation}
    \label{eq:5}
    f(z) = -m - \sum_{k=1}^n m_k \frac{2\lambda_k}{z-\lambda_k} .
  \end{equation}
  It remains to verify that each $m_k >0$ and $m = \sum_k m_k$.  This
  will then imply
  \begin{equation*}
    f(z) = -\sum_{i=1}^n m_i \frac{z+\lambda_i}{z-\lambda_i},
  \end{equation*}
  meaning that $f$ has a Herglotz representation with positive finite
  atomic measure $\mu = \{(\lambda_j, m_j)\}_{j=1}^n$ on $\mathbb T
  \times \mathbb R_{>0}$.

  By \eqref{eq:5}, $\lim_{z\to \lambda_k} (z-\lambda_k)f(z) =
  -2\lambda_k m_k$.  Also, since $\varphi(\lambda_k) = 1$,
  \begin{equation*}
    \lim_{z\to \lambda_k} (z-\lambda_k)f(z) = \lim_{z\to \lambda_k}
    \frac{1 + \varphi(z)}{ \frac{1 - \varphi(z)}{z-\lambda_k}} =
    \frac{-2}{\varphi'(\lambda_k)},
  \end{equation*}
  and so by \eqref{eq:4}
  \begin{equation*}
    m_k = \frac{1}{\lambda_k\varphi'(\lambda_k)} > 0.
  \end{equation*}
  
  The assumptions that $c\in \mathbb T$ and $\varphi(0) = c
  S_n(\alpha) \in \mathbb R$, along with \eqref{eq:5} and
  \eqref{eq:3}, imply that
  \begin{equation*}
    -m = \lim_{z\to\infty} f(z) = \frac{\overline{S_n(\alpha)} +
      c}{\overline{S_n(\alpha)} - c} = -\frac{1 +
      c\overline{S_n(\alpha)}}{1 - c\overline{S_n(\alpha)}} = 
    -f(0) = m - 2\sum_k m_k.
  \end{equation*}
  Hence $m = \sum_k m_k$.  Also, if $\alpha_j = 0$ for some $j$, then
  $m = 1$, and so $\mu$ is a probability measure.  Conversely, if
  $\mu$ is a probability measure, then $c\prod_j \alpha_j = 0$, and so
  $\alpha_j = 0$ for some $j$.
\end{proof}

Recall the assumption that $B'$ is a Blaschke product of degree bigger
than $1$ with a zero at $\alpha_0 = 0$ of multiplicity at least $1$.
Write $t_j$ for the multiplicity of the zero $\alpha_j$ of $B'$.

If $f = M\circ \varphi$ where $\varphi\in H^\infty_{B'}$ with
$\varphi(0) = 0$, then there are constraints imposed on the
corresponding probability measure $\mu$.  For $j>0$,
\begin{equation*}
  1 = f(\alpha_j) = \int_\mathbb T \frac{w+\alpha_j}{w-\alpha_j}\,
  d\mu(w) = \int_\mathbb T \left[ 1 + \frac{2\alpha_j}{w-\alpha_j}
  \right]\, d\mu(w) = 1+ 2\alpha_j\int_\mathbb T
  \frac{1}{w-\alpha_j}\, d\mu(w),
\end{equation*}
and thus
\begin{equation*}
  \int_\mathbb T \frac{1}{w-\alpha_j}\, d\mu(w) = 0,\qquad j > 0.
\end{equation*}

By an induction argument,
\begin{equation*}
  f^{(k)}(z) = 2k!\int_\mathbb T \frac{w}{(w-z)^{k+1}}\, d\mu(w) = 2k!
  \left[ \int_\mathbb T \frac{1}{(w-z)^{k}}\, d\mu(w) + \int_\mathbb T
    \frac{z}{(w-z)^{k+1}}\, d\mu(w)\right].
\end{equation*}
If the multiplicity $t_j$ of the root $\alpha_j$, is bigger than $1$,
then as neither $M$ nor its derivatives have any zeros in $\mathbb D$,
the Fa\`a di Bruno formula implies that
\begin{equation*}
  f^{(k)}(\alpha_j) = 0, \qquad j>0 \text{ and } k=1,\dots t_j-1.
\end{equation*}
Thus
\begin{equation}
  \label{eq:6}
  0 =  \int_\mathbb T \frac{1}{(w-\alpha_j)^{k}}\, d\mu(w), \qquad j>0
  \text{ and } k=1,\dots t_j.
\end{equation}
For $z = \alpha_0 = 0$, if $t_0 > 1$, then
\begin{equation}
  \label{eq:7}
  0 =  \int_\mathbb T \frac{1}{w^{k}}\, d\mu(w), \qquad
  k=1,\dots t_0-1.
\end{equation}
Consequently, the first $t_0-1$ moments of $\mu$ are zero, and other,
more complex constraints are implied by the formulas involving the
other roots.

Conversely, suppose that $\mu$ is a probability measure for which
\eqref{eq:6} and \eqref{eq:7} hold.  If, for example,
\begin{equation*}
  0 = \int_{\mathbb T} \frac{1}{w - \alpha_1}\,d\mu(w),
\end{equation*}
then
\begin{equation*}
  f(\alpha_1) = \int_{\mathbb T} \frac{w + \alpha_1}{w -
    \alpha_1}\,d\mu(w)
  = \int_{\mathbb T} \frac{w - \alpha_1}{w - \alpha_1} \,d\mu(w)
  +2\alpha_1 \int_{\mathbb T} \frac{1}{w - \alpha_1}\,d\mu(w)
  =1.
\end{equation*}
By the same reasoning, $f(\alpha_j) = 1$ for all $j$.  Similar
calculations show that $f^{(k)}(\alpha_j) = 0$ for $1\leq k \leq t_j -
1$.

Denote the set of positive measures satisfying the constraints
in~\eqref{eq:6} and~\eqref{eq:7} by $M^+_{B',\mathbb R}(\mathbb T)$.
This is a weak-$*$ closed, convex, locally compact set in the Banach
space of finite Borel measures $M_{B',\mathbb R}(\mathbb T) =
\overline{\bigvee M^+_{B',\mathbb R}(\mathbb T)}$, and is additionally
a cone since it is closed under sums, positive scalar multiples, and
$M^+_{B',\mathbb R}(\mathbb T) \cap -M^+_{B',\mathbb R}(\mathbb T) =
\{0\}$.  Recall that in a convex set $A$ in a vector space $X$, $E
\subset A$ is an \emph{extremal set} if whenever $a\in E$ and $a = tx
+ (1-t)y$ for $x,y \in A$ and $t\in (0,1)$, it follows that $x,y \in
E$.  One point extremal sets are \emph{extreme points}, while extremal
sets which are half lines are termed \emph{extreme rays}
(\emph{extreme directions} in \cite{MR2375060}).  Here we follow the
conventions laid out in Holmes~\cite{MR0410335}.

Write $M^{+,1}_{B',\mathbb R}(\mathbb T)$ for the probability measures
in $M^+_{B',\mathbb R}(\mathbb T)$.  This set is weak-$*$ closed,
convex, and compact, and forms a \emph{base} for $M^+_{B',\mathbb
  R}(\mathbb T)$, in that any $\tilde\mu \in M^+_{B',\mathbb
  R}(\mathbb T)$ is of the form $t\mu$ for some $\mu\in
M^{+,1}_{B',\mathbb R}(\mathbb T)$ and $t \geq 0$.  By the
Kre{\u\i}n-Mil$'$man theorem, $M^{+,1}_{B',\mathbb R}(\mathbb T)$ is
the closed convex hull of $\hat\Theta$, the set of its extreme points,
and it is an elementary observation that $\mu \in M^{+,1}_{B',\mathbb
  R}(\mathbb T)$ is an extreme point if and only if $\{t\mu:
t\in\mathbb R^+\}$ is an extreme ray in $M^{+}_{B',\mathbb R}(\mathbb
T)$.

By the Choquet-Bishop-de~Leeuw theorem~\cite{MR1835574}, to any
$\mu\in M^{+,1}_{B',\mathbb R}(\mathbb T)$ there corresponds a $\nu$
on $\hat\Theta$ such that
\begin{equation*}
  \mu = \int_{\hat\Theta} \theta \,d\nu_\mu(\theta).
\end{equation*}
This is reminiscent of the Alexandrov disintegration
theorem~\cite{MR2215991}.

For $\mu\in \hat\Theta$, define
\begin{equation*}
  f_\mu(z) := \int_\mathbb T \frac{w+z}{w-z}\, d\mu(w),
\end{equation*}
an analytic function on $\mathbb D$ with positive real part and value
$1$ when $z=0$.  As in \cite{MR2946923}, this yields the so-called
\emph{Agler-Herglotz representation}.

\begin{theorem}[Agler-Herglotz representation associated to $\mathcal
  A_B$]
  \label{thm:Agler-Herglotz-repn-AB}
  Let $f$ be analytic on $\mathbb D$ with positive real part, and
  suppose further that
  \begin{equation}
    \label{eq:8}
    f(\alpha_0) = f(\alpha_j) = 1, \quad \text{ and }\quad
    f^{(k)}(\alpha_j) = 0, \quad j=1,\dots,n, \ 1\leq k \leq t_j - 1.
  \end{equation}
  Then there is a probability measure $\nu$ on the set of extreme
  points $\hat\Theta$ of $M^{+,1}_{B',\mathbb R}(\mathbb T)$ such that
  \begin{equation}
    \label{eq:9}
    f(z) = \int_{\hat\Theta} f_\mu(z) \, d\nu(\mu).
  \end{equation}
\end{theorem}

Next turn to concretely characterizing the elements of $\hat\Theta$.
This is done by first finding a dual characterization of the
constraints in terms of the annihilator of $H^\infty_{B'}$.  The
following is in fact a special case of what is considered by Ball and
Guerra-Huam\'an in~\cite{MR3138369}.  Nevertheless, the special nature
of the algebras considered here allow us to give much more specific
information.

As usual, $L^2_{\mathbb R}(\mathbb T)$ will stand for the Hilbert
space of real valued square integrable functions on the unit circle.
Also, $M_{\mathbb R}(\mathbb T)$ stands for the space of finite
regular real Borel measures on $\mathbb T$, which is the dual of
$C_{\mathbb R}(\mathbb T)$ with the norm topology, as well as being
the weak-$*$ predual of this space.  Every $\mu\in M_{\mathbb
  R}(\mathbb T)$ is associated by means of a Poisson kernel to a
harmonic function $\hat\mu$, which in this setting is the real part of
a holomorphic function on $\mathbb D$.  The space $L^2_{\mathbb
  R}(\mathbb T)$ contains a subspace $L^2_{B',\mathbb R}(\mathbb T)$
consisting of those functions which are the real parts of functions in
$\mathcal A_{B'}$ restricted to $\mathbb T$.

Write $\alpha_0,\dots,\alpha_m$ for the distinct zeros of $B'$ with
respective multiplicities $t_0,\dots,t_m$, $N = \sum t_j$.  Because
each $\alpha_j \in\mathbb D$, a function in $\mathcal A_{B'}$ can be
written as
\begin{equation*}
  \varphi(z) = c + \prod_0^N (z-\alpha_j)^{t_j} g(z),
\end{equation*}
for some $g\in A(\mathbb D)$.  It is a standard result that the
complex annihilator ${\mathcal A_{B'}}^\bot$ is isometrically
isomorphic to the dual of $A(\mathbb D) /\mathcal A_{B'}$.  The latter
space is spanned by $z^k + \mathcal A_{B'}$, $k=1,\dots, N-1$, and so
has dimension $N-1$.  Hence the dimension of ${\mathcal A_{B'}}^\bot$
is also $N-1$.

The kernel functions
\begin{equation}
  \label{eq:10}
  k^{(i)}_\alpha(z) = k^{(i)}(\alpha,z) :=
  \frac{i!z^i}{(1-\overline{\alpha}z)^{i+1}}
\end{equation}
have the property that $\ip{\varphi}{k^{(i)}_\alpha} =
\varphi^{(i)}(\alpha)$, the $i$th derivative of $\varphi$ evaluated at
$\alpha$.  So for $0\leq j \leq m$, $1 \leq i \leq t_j - 1$ and
$\varphi\in \mathcal A_{B'}$,
\begin{equation*}
  \ip{\varphi}{k^{(i)}_\alpha} = 0.
\end{equation*}
This accounts for $-m + \sum t_m = N-m$ linearly independent functions
in the annihilator.  If $m > 0$, fix $\alpha_\ell$.  Then for
$j=1,\dots , m$ and $j\neq \ell$,
\begin{equation*}
  \ip{\varphi}{k^{(0)}_{\alpha_\ell} - k^{(0)}_{\alpha_j}} = c-c =
  0.
\end{equation*}
These $m-1$ functions along with the previous $N-m$ functions then
form a linearly independent set, and hence a basis for the complex
annihilator of $\mathcal A_{B'}$.  By the way, this argument works
even when no $\alpha_i = 0$.  Write $\{g_k\}$ for this set of
functions.

These functions are connected to the constraints constructed above,
since with the measure $\mu$ from the Herglotz representation, there
will be $h_j$ such that
\begin{equation*}
  0 = \ip{\varphi}{g_j} = \int_{\mathbb T} h_j \,d\mu;
\end{equation*}
namely,
\begin{equation*}
  h_j = \ip{\frac{1+\varphi}{1-\varphi}}{g_j}.
\end{equation*}

\begin{lemma}
  \label{lem:annih-A_B}
  Let $B$ be a Blaschke product with zeros $\alpha_0,\dots, \alpha_m$
  with multiplicities $t_0,\dots,t_m$, and set $N = \sum t_j$.  The
  annihilator ${\mathcal A_B}^\bot$ is an $N-1$ dimensional space,
  with basis made up of the functions $k^{(i)}(\alpha_j,\cdot)$, for
  all $0\leq j \leq m$ such that $t_j > 1$ and $1 \leq i \leq t_j -
  1$, as well as $k^{(0)}_{\alpha_0} - k^{(0)}_{\alpha_j}$ in case
  $m>1$ and $j=0,\dots , m$.
\end{lemma}

If $\varphi\in \mathcal A_B$, then both $\mathrm{Re}\,h_k$ and
$\mathrm{Im}\,h_k$ are orthogonal to $\mu$ in $L^2_{B',\mathbb
  R}(\mathbb T)$.  As explained in Section~4.1 of~\cite{MR3138369}
(and generalizing similar results in~\cite{MR2375060}),
\begin{equation*}
  M_{B',\mathbb R}(\mathbb T) = L^2_{B',\mathbb R}(\mathbb T)^\bot = 
  {\{\mathrm{Re}\,h_k,\ \mathrm{Im}\,h_k\}}_{k=1,\dots,N-1}^{\bot},
\end{equation*}
and
\begin{equation*}
  C_{B',\mathbb R}(\mathbb T)^\bot = \mathrm{span}\,
  {\{\mathrm{Re}\,h_k\,ds,\ \mathrm{Im}\,h_k\,ds\}}_{k=1,\dots,N-1},
\end{equation*}
a $(2N-2)$-dimensional space.  Here $ds$ represents arc-length measure
on $\mathbb T$.

\begin{theorem}
  \label{thm:extr-pt-char}
  Let $N$ be the number of zeros of $B'$, counting multiplicities.  If
  $\mu$ is an extreme point of $M^{+,1}_{B',\mathbb R}(\mathbb T)$,
  then it is a probability measure on $\mathbb T$ supported at $k$
  points, where $N \leq k \leq 2N-1$.
\end{theorem}

\begin{proof}
  The idea of the proof for the upper bound is the same as for
  Theorem~5 of \cite{MR2946923} and Lemma~3.5 of \cite{MR2375060}.
  Since the codimension of $M_{B',\mathbb R}(\mathbb T)$ in
  $M_{\mathbb R}(\mathbb T)$ is $2N-2$, if a measure $\mu \geq 0$ is
  supported at $2N$ or more points, $\dim(M_{B',\mathbb R}(\mathbb T)
  \cap M_{\mathbb R}(\mathbb T)) \geq 2$, and so this space contains a
  nonzero measure $\nu \geq 0$ which is linearly independent of $\mu$.
  For small enough $\epsilon > 0$, $\mu \pm \epsilon \nu \geq 0$, and
  then since $\mu$ is a convex combination of these, it is not
  extremal.

  Now consider the lower bound, and suppose $\mu$ is supported at $n <
  N$ points.  By Lemma~\ref{lem:finite-meas-gives-B-prod}, $\mu$ is
  associated to a Blaschke product $\tilde B = \prod_1^n
  \frac{z-\beta_j}{1-\overline{\beta_j}z}$ with $n$ zeros.  Let
  $\alpha_0,\dots,\alpha_m$ be the zeros of $B$ with multiplicities
  $t_0,\dots,t_m$.  Since $\tilde B \in \mathcal A_B$, there is a
  constant $c$ such that $\tilde B(\alpha_j) = c$ for all $j$.  Then
  for each $j$, $\tilde B(z) - c$ is seen to have a zero of
  multiplicity $t_j$ at $\alpha_j$, and so $\tilde B(z) - c$ has at
  least $N$ zeros.  But
  \begin{equation*}
    \tilde B(z) - c = \frac{\prod_1^n(z-\beta_j) - c \prod_1^n(1 -
      \overline{\beta_j}z)}{\prod_1^n(1 - \overline{\beta_j}z)},
  \end{equation*}
  and since the numerator is a polynomial of degree $n<N$, there is a
  contradiction.
\end{proof}

A similar construction can be carried out for $\mathcal A_B^0$.  Since
$\mathcal A_B^0 \subseteq \mathcal A_B$, for $\varphi \in \mathcal
A_B^0$, $f = \frac{1+\varphi}{1-\varphi}$ will have a Herglotz
representation with a measure $\mu$ satisfying the constraints in
\eqref{eq:6} and \eqref{eq:7}, as well as other, more complex
constraints if $N > 2$.

Counting the functions in $A(\mathbb D)$ of the form $z^iB^j$ as given
in Theorem~\ref{thm:algebras-isom-isom}, the dimension of
$A_B^{0\,\bot}$ is found to be $N(N-1)/2$.  A total of $N-1$ of these
are listed in Lemma~\ref{lem:annih-A_B}.  To recover the other
$(N-1)(N-2)/2$ (where without loss of generality it is now assumed
that $N > 2$), first observe that by
Theorem~\ref{thm:algebras-isom-isom}, $\mathcal A_B^0 \supseteq
\mathcal A_{B^{N-1}}$.  So the remaining elements of a basis for
$A_B^{0\,\bot}$ can be expressed in terms of the basis for $\mathcal
A_{B^{N-1}}^\bot$ (as given by Lemma~\ref{lem:annih-A_B}) minus those
in $\mathcal A_B^\bot$ from the same lemma; that is, in terms of
$k^{(i)}(\alpha_j,\cdot)$, $t_j \leq i \leq (N-1)t_j - 1$.  There are
a total of $N(N-2)-m$ such functions.

A general function in $\mathcal A_B^0$ has the form
\begin{equation*}
  \varphi(z) = \sum_{r=0}^{N-2} B(z)^r \sum_{s=0}^r a_{rs} z^s +
  B(z)^{N-1}g(z), \quad g\in A(\mathbb D),\ a_{rs}\in \mathbb C \text{
    for all }r,s.
\end{equation*}
Let $v = (v_{\ell,j_\ell}) \in \mathbb C^{N(N-2)-m}$ be such that for
all $\varphi \in \mathcal A_B^0$,
\begin{equation*}
  \begin{split}
    0 &= \ip{\varphi}{\sum_{\ell = 1}^m \sum_{j_\ell =
        t_\ell}^{(N-1)t_\ell - 1} v_{\ell,j_\ell}
      k_{\alpha_\ell}^{(j_\ell)}} \\
    & = \sum_{\ell = 1}^m \sum_{j_\ell = t_\ell}^{(N-1)t_\ell - 1}
    \overline{v_{\ell,j_\ell}} \sum_{r=0}^{N-2} \sum_{s=0}^r
    \sum_{q=0}^{j_\ell} a_{rs} \binom{j_\ell}{q} B^{r\,(j_\ell -
      q)}(\alpha_\ell) \alpha_\ell^{s\,(q)} \\
    & = \sum_{r=0}^{N-2} \sum_{s=0}^r a_{rs} \left[\sum_{\ell = 1}^m
      \sum_{j_\ell = t_\ell}^{(N-1)t_\ell - 1}
      \left(\sum_{q=0}^{j_\ell} B^{r\,(j_\ell - q)}(\alpha_\ell)
        \alpha_\ell^{s\,(q)}\right) \overline{v_{\ell,j_\ell}}
    \right],
  \end{split}
\end{equation*}
with the shorthand notation $B^{r\,(j_\ell - q)}(\alpha_\ell) =
\frac{d^{j_\ell - q}}{dz^{j_\ell - q}} B^r(z) |_{\alpha_\ell}$ and
$\alpha_\ell^{s\,(q)} = \frac{d^q}{dz^q} z^s |_{\alpha_\ell}$.  Define
$N(N-1)/2$ vectors $b^{rs} = (b^{rs}_{\ell,j_\ell}) \in \mathbb
C^{N(N-2)-m}$ by
\begin{equation}
  \label{eq:11}
  \begin{split}
    b^{rs}_{\ell,j_\ell} & = \sum_{q=0}^{j_\ell} B^{r\,(j_\ell -
      q)}(\alpha_\ell) \alpha_\ell^{s\,(q)} \\
    &=
    \begin{cases}
      \sum_{q=0}^{\min\{s, j_\ell - rq\}}
      \frac{s!}{(s-q)!}\binom{j_\ell}{q} B^{r\,(j_\ell -
        q)}(\alpha_\ell) \alpha_\ell^{s-q} & j_\ell - rq \geq 0, \\
      0 & \text{otherwise,}
    \end{cases}
  \end{split}
\end{equation}
using the fact that $B^{r\,(n)}(\alpha_\ell) = 0$ if $n < rt_\ell$ for
the last equality.  Accordingly, there will be $(N-1)(N-2)/2$ nonzero,
linearly independent vectors $v$ such that $\ip{b^{rs}}{v} = 0$.

As a simple illustration of this, suppose $B(z) = z^N$, $N > 2$.  Thus
$m = \ell = 1$, $t_\ell = N$ and $N \leq j_1 \leq N(N-1) - 1$.
Nonzero entries in $b^{rs}$ require $j_\ell = rN+t$ and $s = t$, so
$rN \leq j_\ell \leq N(N-1) - 1$.  Clearly, if $r = 0$, then $s = 0$
and so $b^{00}_{1,j_\ell} = 0$ for all $j_\ell$.  On the other hand,
for $r>0$, it follows that since $s \leq r \leq N-2$, $j_\ell -
rt_\ell \geq (N-1)r$.  Thus the only nonzero term in the last sum in
\eqref{eq:11} occurs when $t = s$.  Correspondingly, $j_\ell = rN +
s$, in which case the vector $b^{rs}$ has all entries equal to $0$
except the $(1,rN + s)$ entry, which equals $s!\binom{rN + s}{s} =
\frac{(rN+s)!}{(rN)!}$.  It is then a straightforward exercise to
choose the set of $(N-1)(N-2)/2$ linearly independent vectors $v$
orthogonal to the vectors $b^{rs}$.  For example, take $k^{(i)}_0(z) =
i!z^i$ with $i\neq rN+s$, $1 \leq r \leq N-2$ and $0 \leq s \leq r$.

In fact this can be seen more directly, since a typical element of
$\mathcal A_B^0$ has the form
\begin{equation*}
  \sum_{r=0}^{N-2} \sum_{s=0}^r a_{rs} z^{rN+s} +
  B(z)^{N-1}g(z), \quad g\in A(\mathbb D),\ a_{rs}\in \mathbb C \text{
    for all }r,s.
\end{equation*}
The $N-1$ basis elements contributed from $\mathcal A_B^\bot$ have the
form $k^{(i)}_0(z) = i!z^i$, where $1 \leq i \leq N-1$.  Skip the
$N$th and $(N+1)$st $k^{(i)}_0$ since these are not orthogonal to
$B(z) = z^N$ and $zB(z) = z^{N+1}$.  However, $k^{(i)}_0$, $N+2 \leq i
\leq 2N-1$, will give $0$, and there are $N-2$ of these.  Continue in
this fashion, with the last basis element being $k^{(N(N-1)-1)}_0$,
for a total of $N(N-1)/2$ basis vectors.

As it happens, not all of the functions in
\begin{equation*}
  \{cB^rz^s: 1 \leq r \leq N-2,\ 0 \leq s \leq r,\
  cB^r(1) =1\} \cup
  \{cB^{N-1}m_\alpha : \alpha \in \hat{\mathbb D},\
  cB^{N-1}(1)m_\alpha(1) = 1\},
\end{equation*}
will be needed to form a set of test functions.  Since for $r \geq s$,
\begin{equation*}
  \begin{split}
    & 1-z^sB^r(z)B^{*r}(w)w^{*s}\\ =&\, 1-B^{r-s}(z)B^{* r-s}(w) +
    B^{r-s}(z)(1-z^sB^s(z)B^{*s}(w)w^{*s})B^{* r-s}(w) \\
    =&\, (1-B(z)B^*(w)) +B(z)(1-B(z)B^*(w))B^*(w) + \\
    &\quad \cdots +
    B(z)^{r-s-1}(1-B(z)B^*(w))B^{*r-s-1}(w) + \\
    +& B(z)^{r-s}\left((1-zB(z)B^*(w)w^*)B(w)^{*r-s} +
      zB(z)(1-zB(z)B^*(w)w^*)B^*(w)w^* + \right. \\
    &\quad \left.\cdots +
      z^{s-1}B(z)^{s-1}(1-zB(z)B^*(w)w^*)B^{*s-1}(w)w^{*s-1} \right)
    B^{*r-s}(w)\\
    =&\, h_B(z)(1-B(z)B^*(w))h_B^*(w) +
    h_{zB}(z)(1-zB(z)B^*(w)w^*)h_{zB}^*(w),
  \end{split}
\end{equation*}
the set
\begin{equation}
  \label{eq:12}
  \{B,zB\} \cup \{cB^{N-1}m_\alpha : \alpha \in \hat{\mathbb D},\
  cB^{N-1}(1)m_\alpha(1) = 1\},
\end{equation}
will suffice.

As before, equate the elements of the basis for ${\mathcal
  A_B}^{0\,\bot}$ with a set of constraints on a probability measure
$\mu$.  Let $\mathcal R$ denote the collection of constraints.  The
elements of $\mathcal R$ will involve not only terms like those given
in \eqref{eq:6} and \eqref{eq:7}, but also linear combinations of such
terms.  In the example where $B(z) = z^N$, the constraints have the
form given in \eqref{eq:7}, but now with $k \in \{0,\dots,N-1\}\cup
\{N+2,\dots,2N-1\} \cup \cdots \cup \{N(N-1)-1\}$.

A probability measure $\mu$ satisfying the constraints in $\mathcal R$
gives rise via the Herglotz representation to an analytic function $f$
on $\mathbb D$ which has positive real part and equals $1$ at $0$.  A
Cayley transform of $f$ then yields an element $\varphi$ of the
algebra $\mathcal A_B^0$ which is zero at $0$.  Conversely, every such
element of the algebra gives rise to a probability measure satisfying
these constraints.  Write $M^{+,1}_{B',0,\mathbb R}(\mathbb T)$ for
the set of all such measures.  This is a compact, convex set.  Denote
by ${\hat\Theta}^0$ the extreme points of this set of measures.

The proofs of the next three results mimic those given earlier in the
context of $\mathcal A_B$, and so are omitted.

\begin{theorem}[Agler-Herglotz representation associated to $\mathcal
  A_B^0$]
  \label{thm:Agler-Herglotz-repn-AB0}
  Let $f$ be analytic on $\mathbb D$ with positive real part, and
  suppose further that $\varphi = \frac{1-f}{1+f} \in \mathcal A_B$
  with $\varphi(\alpha_j) = 0$, $1 \leq j \leq m$.  Then there is a
  probability measure $\nu$ on the set of extreme points
  ${\hat\Theta}^0$ of $M^{+,1}_{B',0,\mathbb R}(\mathbb T)$ such that
  \begin{equation*}
    f(z) = \int_{{\hat\Theta}^0} f_\mu(z) \, d\nu(\mu).
  \end{equation*}
\end{theorem}

\begin{lemma}
  \label{lem:annih-A_B0}
  Let $B$ be a Blaschke product with zeros $\alpha_0,\dots, \alpha_m$
  with multiplicities $t_0,\dots,t_m$, and set $N = \sum t_j$.  The
  annihilator ${\mathcal A_B}^{0\,\bot}$ is an $N(N-1)/2$ dimensional
  space containing ${\mathcal A_B}^\bot$ and contained in the
  annihilator ${\mathcal A_{B^{N-1}}}^\bot$.
\end{lemma}

\begin{theorem}
  \label{thm:extr-pt-A_B0-char}
  Let $N$ be the number of zeros of $B'$, counting multiplicities.  If
  $\mu$ is an extreme point of $M^{+,1}_{B',0,\mathbb R}(\mathbb T)$,
  then $\mu$ is probability measure on $\mathbb T$ supported at $k$
  points, where $N \leq k \leq N(N-1)+1$.
\end{theorem}

Now translate our results on measures to statements about functions in
the unit balls of $\mathcal A_{B'}$ and $\mathcal A^0_{B'}$.  Using a
Cayley transform from the right half plane to the unit disk, for each
$\mu\in M^{+,1}_{B',\mathbb R}(\mathbb T)$ (respectively,
$M^{+,1}_{B',0,\mathbb R}(\mathbb T)$), define a map from $\mathbb D$
to itself
\begin{equation*}
  \psi_\mu := \frac{f_\mu-1}{f_\mu+1},
\end{equation*}
where $f_\mu$ is the function coming from the Herglotz representation
corresponding to $\mu$.  Then
\begin{equation*}
  1-\psi_\mu(z)\psi_\mu(w)^* = 2\frac{f_\mu(z) +
    f_\mu(w)^*}{(f_\mu(z)+1)(f_\mu(w)^*+1)}.
\end{equation*}

If $\mu$ is an extremal measure in either $M^{+,1}_{B',\mathbb
  R}(\mathbb T)$ or $M^{+,1}_{B',0,\mathbb R}(\mathbb T)$, by
Theorem~\ref{thm:extr-pt-char}, it is a finitely supported atomic
probability measure on $\mathbb T$.  It then follows from
Lemma~\ref{lem:finite-meas-gives-B-prod} that for $\mathcal A_B$,
$\psi_\mu = B'R_\mu$, where $R_\mu$ is a Blaschke product, with number
of zeros between $0$ and $N-1$ (where a Blaschke product with no zeros
is taken to be the constant $1$).  For $\mathcal A_B^0$, $R_\mu$ will
be a Blaschke product with the number of zeros is between $0$ and
$N(N-1)+1 -N = (N-1)^2$.  It is evident from the form of elements of
$\mathcal A_B^0$ given in \eqref{eq:12} that any Blaschke product
corresponding to a measure with between $N$ and $N(N-1)$ zeros is of
the form $B$, $zB$ or $B^{N-1}$.  If it has $N(N-1)+1$ zeros, it is of
the form $B^{N-1} m_\alpha$, where $m_\alpha =
\frac{z-\alpha}{1-\overline{\alpha}z}$.

In both cases, the support of the measure $\mu$ corresponds to the set
$\psi_\mu^{-1}(1)$.  Ultimately, a subset of such functions will be
used as test functions.  By the realization theorem, any test function
can be replaced by a unimodular constant times the test function.  So
for convenience, identify $\psi_\mu(z)$ with
$\overline{\psi_\mu(1)}\psi_\mu(z)$.  This amounts to ensuring that
$1$ is one of the support points for the measure $\mu$.  Let $\Theta$
(respectively, $\Theta^0$) be the subset of measures in $\hat\Theta$
(respectively, ${\hat\Theta}^0$) having $1$ as a support point, and
write $\Psi_{B'}$ (respectively, $\Psi_{B'}^0$) for the collection
$\{\psi_\mu\}_{\mu\in \Theta}$ (respectively, $\{\psi_\mu\}_{\mu\in
  \Theta^0}$).  Then $\Psi_{B'}$ is a set of test functions for
$\mathcal A_{B'}$ and $\Psi_{B'}^0$ is a set of test functions for
$\mathcal A_{B'}^0$.

Suppose that $\varphi$ is in the unit ball of $\mathcal A_{B'}$ with
$\varphi(0) = 0$.  Then $f = M\circ \varphi$ is a holomorphic function
from $\mathbb D$ to $\mathbb H$ for which \eqref{eq:8} holds.  Also,
\begin{equation*}
  \varphi = \frac{f-1}{f+1}.
\end{equation*}
Hence
\begin{equation*}
  1 - \varphi(z)\varphi(w)^* =
  2\frac{f(z)+f(w)^*}{(f(z)+1)(f(w)^*+1)}.
\end{equation*}
As noted above, this in particular holds when $\varphi = \psi_\mu$ and
$f = f_\mu$.  Applying the Agler-Herglotz representation
(Theorem~\ref{thm:Agler-Herglotz-repn-AB}), there is a probability
measure $\nu$ on $\hat\Theta$ such that \eqref{eq:9} holds.  Thus
\begin{equation}
  \label{eq:13}
  \begin{split}
    1 - \varphi(z)\varphi(w)^* &= \frac{2}{(f(z)+1)(f(w)^*+1)}
    \int_{\hat\Theta} (f_\mu(z) +
    f_\mu(w)^*) \, d\nu(\mu) \\
    &= \int_{\hat\Theta} H_\mu(z) (1-\psi_\mu(z)\psi_\mu(w)^*)
    H_\mu(w)^* \, d\nu(\mu),
  \end{split}
\end{equation}
$H_\mu = \frac{(f_\mu + 1)}{f+1}$.  It follows that
\begin{equation*}
  1 - \varphi(z)\varphi(w)^* = \Gamma(z,w)(1-E(z)E(w)^*),
\end{equation*}
with $\Gamma : \mathbb D\times \mathbb D \to C(\Psi_{B'})$ the
positive kernel given by
\begin{equation*}
  \Gamma(z,w)g = \int_{\hat\Theta} H_\mu(z)
  g(\psi_\mu) H_\mu(w)^* \, d\nu(\mu).
\end{equation*}

More generally, if $\varphi(0) =c \neq 0$, define
\begin{equation*}
  \tilde\varphi(z) = \frac{\varphi(z) - c}{1-
    \overline{c}\varphi(z)}.
\end{equation*}
Then
\begin{equation*}
  1 - \tilde\varphi(z)\overline{\tilde\varphi(w)} =
  \frac{(1-c\overline{c})(1 -
    \varphi(z)\overline{\varphi(w)})}{(1-\overline{c}\varphi(z))
    (1-c\overline{\varphi(w)})}.
\end{equation*}
Now define $\Gamma$ as before, but with
\begin{equation*}
  H_\mu = \frac{\sqrt{1-c\overline{c}}(1 - \overline{c}\varphi)
    (f_\mu +1)}{\tilde f + 1},
\end{equation*}
where $\tilde f = M\circ \tilde\varphi$, and $\nu$ is chosen as the
probability measure associated to $\tilde\varphi$ in the
Agler-Herglotz theorem.

Once again, the same arguments work with $\mathcal A_B^0$ in place of
$\mathcal A_B$.  Combining Theorem~\ref{thm:extr-pt-char} with
Lemma~\ref{lem:finite-meas-gives-B-prod} yields the following.

\begin{corollary}
  \label{cor:extremals-are-B-prods}
  Let $\hat\Theta$ be the set of extreme measures in
  $M^{+,1}_{B',\mathbb R}(\mathbb T)$.  Then $\hat\Theta$ is a subset
  of the atomic probability measures supported at $N \leq k \leq 2N-1$
  points in $\mathbb T$, where $N$ is the number of zeros of $B'$,
  counting multiplicity.  Furthermore, the set
  \begin{equation*}
    \left\{\psi_\mu = cBR_\mu : R_\mu\text{ a Blaschke product with
        between }0\text{ and }N-1\text{ zeros},\
      c = \prod\frac{1-\overline{\alpha_j}}{1-\alpha_j} \right\}.
  \end{equation*}
\end{corollary}

\begin{corollary}
  \label{cor:extremals-are-B-prods-A_B0}
  Let ${\hat\Theta}^0$ be the set of extreme measures in
  $M^{+,1}_{B',0,\mathbb R}(\mathbb T)$.  Let $N$ be the number of
  zeros of $B'$ counting multiplicities.  Then ${\hat\Theta}^0$ is a
  subset of the atomic probability measures supported at $rN +s$
  points in $\mathbb T$, $1 \leq r \leq N-2$, $0 \leq s \leq r$, which
  are associated to the Blaschke products $B^rz^s$, as well as those
  coming from $B^{N-1}$ or $B^{N-1}m_\alpha$, $m_\alpha$ a M\"obius
  map.  Furthermore,
  \begin{equation*}
    \Psi_{B'}^0 = \left\{cBz^s: s\in \{0,1\},\
      cB(1) =1\right\} \cup
      \left\{cB^{N-1}m_\alpha : \alpha \in \hat{\mathbb D},\
      cB^{N-1}(1)m_\alpha(1) = 1\right\},
  \end{equation*}
  is a collection of test functions for $\mathcal A_{B'}^0$, where
  $\hat{\mathbb D} = \mathbb D \cup \{\infty\}$ is the one point
  compactification of $\mathbb D$ and $m_\infty = 1$.
\end{corollary}

In Corollary~\ref{cor:extremals-are-B-prods-A_B0}, ${B'}^{N-1}$ and
${B'}^{N-1}z$ can obviously be removed from the set of test functions
(or indeed, any countable subset of $\{{B'}^{N-1}m_\alpha\}$), but
then the set would no longer be compact.

Applying a M\"obius map if needed gives sets of test functions
$\Psi_B$ for $\mathcal A_B$ and $\Psi_B^0$ for $\mathcal A_B^0$.  The
sets have exactly the same form, but with $B'$ replaced by $B$.  This
is clear for $\mathcal A_B$.  Let us check it for $\mathcal A_B^0$.
Assume $\alpha$ is a zero of $B$, and let $B' = B\circ m_{-\alpha}$,
which has a zero at $0$.  Set $x' = B'$ and $y' = B'z'$, where $z' =
m_{-\alpha}$, and use these in defining $P'$, the polynomial whose
zero set is the variety for $B'$.  The algebra $\mathcal A_{B'}$ has a
set of test functions $\Psi_{B'}^0$ as described in
Corollary~\ref{cor:extremals-are-B-prods-A_B0}.  These are mapped
completely isometrically isomorphically onto a set of test functions
$\Psi_{\mathcal N(B')}$ for $A(\mathcal N_{B'})$.  Set $x = x'\circ
m_\alpha = B$ and $y = y'\circ m_\alpha = Bz$.  This maps
$\Psi_{\mathcal N_{B'}}$ to $\Psi_{\mathcal N_B}$, a set of test
functions for $A(\mathcal N_B)$.  Setting $z = y/x$ identifies these
with a set of test functions of $\mathcal A_B^0$.  Combining these
gives a map from $\Psi_{B'}^0$ to $\Psi_B^0$ which takes $B'$ to $B$,
$B'z'$ to $Bz$ and $\{B'^{N-1}m_\alpha: \alpha \in \hat{\mathbb D}\}$
to $\{B^{N-1}m_\alpha: \alpha \in \hat{\mathbb D}\}$.

There is then a corresponding collection of admissible kernels
$\mathcal K_{\Psi_{B}}$ (respectively, $\mathcal K_{\Psi_{B}^0}$), and
function algebra $H^\infty(\mathcal K_{\Psi_{B}})$ (respectively,
$H^\infty(\mathcal K_{\Psi_{B}^0})$).

\begin{theorem}
  \label{thm:algs_isom_isom}
  The algebras $H^\infty(\mathcal K_{\Psi_{B}})$ and $H^\infty_{B}$
  are isometrically isomorphic.  Likewise, the algebras
  $H^\infty(\mathcal K_{\Psi_{B}^0})$ and $H^{\infty\, 0}_{B}$ are
  isometrically isomorphic.
\end{theorem}

\begin{proof}
  The difficult part of the proof has been done above.  It is simply
  left to note that since any test function is in the unit ball of
  $H^\infty(\mathbb D)$, the Szeg\H{o} kernel $k_s$ is an admissible
  (for both $H^\infty(\mathcal K_{\Psi_B})$ and $H^\infty(\mathcal
  K_{\Psi_B^0})$).  Hence for any function $\varphi$ in the unit ball
  of either of these algebras, $((1-\varphi(x)\varphi(y)^*)k_s(y,x))$
  is a positive kernel, and so $\varphi$ is in the unit ball of
  $H^\infty_B$ (respectively, $H^{\infty\, 0}_B$).
\end{proof}

\section{Minimality of the set of test functions}
\label{sec:descr-test-funct}

At this point, Corollaries~\ref{cor:extremals-are-B-prods}
and~\ref{cor:extremals-are-B-prods-A_B0} give a fairly concrete
description of a set of test functions, especially for $\mathcal
A_{B'}^0$.  However, in dealing with $\mathcal A_{B'}$, it is more
useful for what follows to describe the test functions in terms of the
placement of the zeros rather than the support points for the measure
in the Herglotz representation.  Obviously, in writing any test
function as a Blaschke product, changing the order of the zeros does
not change the function.  There is also the point that the number of
zeros of a test function for $\mathcal A_{B'}$ is between $N$ and
$2N-1$, where $N$ is the number of zeros of $B'$, so not all test
functions will have the same number of zeros.

Introduce the following order on elements of the one point
compactification of the disk, $\hat{\mathbb D} = \mathbb D \cup
\infty$, which take these considerations into account: $\zeta_1
\preceq \zeta_2$ in $\hat{\mathbb D}$ if either $|\zeta_1| <
|\zeta_2|$ or $|\zeta_1| = |\zeta_2|$ and $\arg\zeta_1 \leq
\arg\zeta_2$.  The point $\infty$ is the maximal element of
$\hat{\mathbb D}$ with respect to this order, and $0$ the minimal
element.

This order can be used to describe the set of test functions for
$\mathcal A_{B'}$.  Let $\zbp = \{{\alpha'}_0 \leq \cdots \leq
{\alpha'}_{N-1}\}$ be the (ordered) zeros of $B'$.  So $B' =
m_{{\alpha'}_0} \cdots m_{{\alpha'}_{N-1}}$, where $m_{{\alpha'}_j}$
is the M\"obius map with zero ${\alpha'}_j$.  If as an abuse of
notation $m_\infty(z) = 1$, then a Blaschke product ${B'}_\alpha =
B'R$ with between $N$ and $2N-1$ zeros can be written as
\begin{equation*}
  {B'}_\alpha(z)= \prod_{j=0}^{2N-2} m_{\alpha_j},
\end{equation*}
where $\mathcal Z({B'}_\alpha) = \{0 = \alpha_0 \preceq \cdots \preceq
\alpha_{2N-2}\}$, the ordered zeros of ${B'}_\alpha$ in $\hat{\mathbb
  D}$, contains $\zbp$.

The set
\begin{equation*}
  \{ c {B'}_\alpha : \mathcal Z({B'}_\alpha) \text{ an ordered } 2N-1
  \text{ tuple containing the elements of } \zbp \text{ and } c =
  \overline{{B'}_\alpha(1)}\}
\end{equation*}
contains $\Psi_{B'}$.  On the other, by
Corollary~\ref{cor:extremals-are-B-prods}, any function in this set
has a corresponding Herglotz representation with between $N$ and
$2N-1$ support points for the measure, the measure satisfies the
constraints \eqref{eq:6} and \eqref{eq:7} since its zero set contains
the zeros of $B'$, and the constant $c$ is chosen so that one of the
support points is at $1$.  Therefore the opposite containment holds.

With this identification, view the measure in \eqref{eq:13} as being
on the set $\Psi_{B'}$ in place of the set of extremal measures
$\hat\Theta$, so that
\begin{equation}
  \label{eq:14}
  1 - \varphi(z)\varphi(w)^* = \int_{\Psi_{B'}} H_\psi(z)
  (1-\psi(z)\psi(w)^*) H_\psi(w)^* \, d\nu(\psi).
\end{equation}
There is an obvious version of this for the algebra $\mathcal
A_{B'}^0$, with $\Psi_{B'}$ replaced by $\Psi_{B'}^0$.

\begin{theorem}
  \label{thm:minimal_set_of_test_fns}
  The set $\Psi_{B'}$ is a minimal set of test functions for the
  algebra $\mathcal A_{B'}$.
\end{theorem}

\begin{proof}
  The set $\Psi_{B'}$ is norm closed in $H^\infty(\mathbb D)$.  Endow
  it with the relative topology, as described in~\cite{MR2389623}.
  Suppose that some proper closed subset $\mathcal C$ of $\Psi_{B'}$
  is a set of test functions for $\mathcal A_{B'}$.  Then $\Psi_{B'}
  \backslash \mathcal C$ is relatively open, and some $\psi_0 = c_0
  \prod_{j=0}^{2N-2} m_{\tilde{\alpha}_j}$ is in this set, where $\zbp
  = \{0 = \tilde{\alpha}_0,\tilde{\alpha}_{j_1} \dots ,
  \tilde{\alpha}_{j_{N-1}}\}$ are the zeros of $B'$.  Since $\Psi_{B'}
  \backslash \mathcal C$ is relatively open, assume without loss of
  generality that no ${\tilde\alpha}_j = \infty$ and that any zero
  which is not a zero for $B'$ is distinct from the zeros of $B'$ and
  all such zeros are distinct from each other.

  Let $\psi = c \prod_{k=0}^{2N-2} m_{\alpha_j}$ be in $\mathcal C$.
  For any $\alpha_k$ in $\mathcal Z(\psi)$ which occurs only once, set
  $k_{\alpha_j}(z) = 1/(1-\overline{\alpha_j}z)$, the Szeg\H{o}
  kernel, where $k_\infty := 0$.  More generally, if $\alpha\neq
  \infty$ is repeated, it is understood that the kernels
  $k^{(i)}_\alpha(z) =i! z^i / (1 - \overline{\alpha} z)^{i+1}$ are
  used instead, where $i$ runs from $0$ to one less than the
  multiplicity of the root, though this is generally not written
  explicitly to avoid notational complexity.  Define
  $k_{\tilde{\alpha}_{j}}$ in an identical manner.

  To prove the theorem, argue by contradiction.  To begin with, by
  the same reasoning to that found in the proof of Theorem~9
  of~\cite{MR2946923}, for $\psi\in\mathcal C$ and $1 \leq \ell \leq
  2N$, there exist functions $h_{\psi,\ell} \in L^2(\nu)$ such that
  \eqref{eq:14} can be written as
  \begin{equation}
    \label{eq:15}
    1 - \psi_0(z)\psi_0(w)^* = \int_{\mathcal C} \sum_{\ell = 1}^{2N}
    h_{\psi,\ell} (z) (1-\psi(z)\psi(w)^*) h_{\psi,\ell} (w)^* \,
    d\nu(\psi).
  \end{equation}
  Furthermore, for $n=0,\dots,2N-2$, there are constants $c_{nj}$ such
  that
  \begin{equation}
    \label{eq:16}
    h_{\psi,\ell} k_{\alpha_n} = \sum_{j=0}^{2N-2} c_{nj}\,
    k_{\tilde{\alpha}_{j}}.
  \end{equation}
  In particular, taking $n=0$ gives
  \begin{equation*}
    h_{\psi,\ell} = \sum_{j=0}^{2N-2} c_{0j}\,
    k_{\tilde{\alpha}_{j}}.
  \end{equation*}
  The kernels extend to meromorphic functions on the Riemann sphere,
  as then does $h_{\psi,\ell}$.
  Plug this last identity back into \eqref{eq:16}, for $n>1$, to get
  \begin{equation}
    \label{eq:17}
    k_{\alpha_n} \sum_{j=0}^{2N-2} c_{0j}\, k_{\tilde{\alpha}_{j}} =
    \sum_{j=0}^{2N-2} c_{nj}\, k_{\tilde{\alpha}_{j}}.
  \end{equation}

  Now use \eqref{eq:17} to eliminate some of the terms and to
  eventually solve for $h_{\psi,\ell}$.  If $\alpha_n \notin \mathcal
  Z(\psi_0)$, the left side of \eqref{eq:16} has a pole at
  $1/\overline{\alpha_n}$, while the right side does not.  In this
  case the only possibility is for $h_{\psi,\ell} = 0$.  Also, if
  $\alpha$ is a zero of multiplicity $t_j$ in $\psi$ and ${\tilde
    t}_j$ in $\psi_0$, then by \eqref{eq:16} with $k_{\alpha_n}$ equal
  to $k^{(t_j-1)}_\alpha(z)$, it follows from \eqref{eq:16} by
  counting pole multiplicities that $t_j \leq {\tilde t}_j$.

  Since the number of zeros of $\psi_0 = 2N-1$ and is greater than or
  equal to the number of zeros of $\psi$, if the two have the same
  number of zeros, they are equal (up to multiplicative unimodular
  constant), which cannot happen.  Hence $\psi$ must have fewer than
  $2N-1$ zeros.

  Consider $0 \neq \alpha_n \in \zbp$.  Then $\alpha_n =
  {\tilde\alpha}_j$ for some $j$.  If this is a zero of order $1$ for
  $\psi_0$, then the right side of \eqref{eq:17} has a pole of order
  at most $1$ at $1/\overline{{\tilde\alpha}_j}$, while the left side
  has a pole of order $2$ at this point if $c_{0j} \neq 0$.  Hence
  $c_{0j} =0$.

  More generally, suppose that $\psi_0$ has a zero of order $m>1$ at
  $\alpha_n \in \zbp$ (where now $\alpha_n$ may be $0$).  Let
  ${\tilde\alpha}_j=\dots ={\tilde\alpha}_{j+m-1}$ be the $m$ repeated
  zeros.  If $\alpha_n \neq 0$, each $k_{{\tilde\alpha}_{j+i}}$,
  $0\leq i \leq m-1$, has a pole of order between $1$ and $m$, and so
  no term on the right side of \eqref{eq:17} has a pole of order more
  than $m$ at $1/\overline{{\tilde\alpha}_j}$.  On the left side, if
  $k_{\alpha_n} = k^{(m-1)}_\alpha$ (which has a pole of order $m$)
  and if any of $c_{0j}$ to $c_{0,j+m-1}$ are nonzero, the
  corresponding term has a pole of order bigger than $m$.  Hence each
  of these coefficients must be zero.

  Things are slightly different when $\alpha_n = 0$.  In this case,
  $j=0$ and each $k_{{\tilde\alpha}_i}$, $1\leq i \leq m-1$, has a
  pole of order between $1$ and $m-1$ at $\infty$ (note that
  $k_{{\tilde\alpha}_0} = 1$).  So reasoning as before, no term on the
  right of \eqref{eq:17} has a pole of order bigger than $m-1$ at
  $\infty$, while $k_{\alpha_n}$ has a pole of order $m-1$ there, the
  left side has a pole of order at least $m$ at $\infty$ if any of
  $c_{0,1}$ to $c_{0,m-1}$ are nonzero.  So all of these coefficients
  must also be zero.

  Let $\alpha_n \in \mathcal Z(\psi)\backslash \zbp \subset \mathcal
  Z(\psi_0)\backslash \zbp$.  By assumption all such zeros are of
  order $1$.  Once again, a pole count with \eqref{eq:16} gives that
  the corresponding coefficient in $h_{\psi,\ell}$ is $0$.

  Combine these observations to conclude that
  \begin{equation*}
    h_{\psi,\ell} = c_{00} + \sum_{{\tilde\alpha}_j \in
      \mathcal Z(\psi_0)\backslash \mathcal Z(\psi)} c_{0j}
    k_{{\tilde\alpha}_j} = g_{\psi,\ell} \prod_{{\tilde\alpha}_j \in
      \mathcal Z(\psi_0)\backslash \mathcal Z(\psi)}
    k_{{\tilde\alpha}_j}.
  \end{equation*}
  Recall that the elements of $\mathcal Z(\psi_0)\backslash \zbp
  \supset \mathcal Z(\psi_0)\backslash \mathcal Z(\psi)$ are distinct
  and none are repeats of elements of $\zbp$.  Consequently,
  \begin{equation*}
    g_{\psi,\ell}(z) = c_{00} \prod_{{\tilde\alpha}_j \in \mathcal
      Z(\psi_0)\backslash \mathcal Z(\psi)}
    (1-\overline{{\tilde\alpha}_j}z) + \sum_{{\tilde\alpha}_j \in
      \mathcal Z(\psi_0)\backslash \mathcal Z(\psi)} c_{0j}
    \prod_{{\tilde\alpha}_n \in  \mathcal Z(\psi_0)\backslash 
      \mathcal Z(\psi),\, n\neq j} (1-\overline{{\tilde\alpha}_n}z)
  \end{equation*}
  is a polynomial of degree at most $N-1$.  So \eqref{eq:17} becomes
  \begin{equation}
    \label{eq:18}
    g_{\psi,\ell} k_{\alpha_n} \prod_{{\tilde\alpha}_j \in \mathcal
      Z(\psi_0)\backslash \mathcal Z(\psi)} k_{{\tilde\alpha}_j} =
    \sum_{j=0}^{2N-2} c_{nj}\, k_{\tilde{\alpha}_{j}}.
  \end{equation}

  Since by assumption $B'$ has a zero at $0$ of degree $t_0 \geq 1$,
  the right side of \eqref{eq:18} has a pole at $\infty$ of order at
  most $t_0 -1$ corresponding to $k_{\alpha_n} = k_0^{(t_0-1)}$.  On
  the other hand, with this choice of $k_{\alpha_n}$, the left side of
  \eqref{eq:18} has a pole at $\infty$ of order $\deg g_{\psi,\ell} +
  t_0-1$.  Hence $g_{\psi,\ell}$ is a constant, and so
  \begin{equation}
    \label{eq:19}
    h_{\psi,\ell} = g_{\psi,\ell} \frac{\prod_{{\tilde\alpha}_j
        \in \mathcal Z(\psi_0)\backslash \zbp}
      k_{{\tilde\alpha}_j}}{\prod_{\alpha_n \in \mathcal
        Z(\psi)\backslash \zbp} k_{\alpha_n}}, \qquad g_{\psi,\ell}
    \in \mathbb C.
  \end{equation}

  Substitute the formula for $h_{\psi,\ell}$ from \eqref{eq:19} into
  \eqref{eq:15}, multiply by $\prod_{{\tilde\alpha}_j \in \mathcal
    Z(\psi_0) \backslash \zbp} (1 - \overline{{\tilde\alpha}_j}z)(1 -
  \overline{{\tilde\alpha}_j}z)^*$ and use the identities in
  \eqref{eq:1} to get
  \begin{equation}
    \label{eq:20}
    \begin{split}
      &\sum_{m,n = 0}^{N-1} \left[z^m{\overline{w}}^n
        S_m(\overline{\tilde\alpha})S_n(\overline{\tilde\alpha})^* -
        B'(z)B'(w)^* z^{N-1-m}{\overline{w}}^{N-1-n} S_m(\tilde\alpha)
        S_n(\tilde\alpha)^*\right] \\
      = & \int_{\mathcal C} \sum_\ell \sum_{m,n = 0}^{\deg\psi}
      |g_{\psi,\ell}|^2 \left(z^m{\overline{w}}^n
        S_m(\overline{\alpha_\psi}) S_n(\overline{\alpha_\psi})^* -
        B'(z)B'(w)^* z^{N-1-m}{\overline{w}}^{N-1-n} S_m(\alpha_\psi)
        S_n(\alpha_\psi)^* \right) \, d\nu(\psi).
    \end{split}
  \end{equation}
  As a shorthand notation, $\tilde\alpha$ stands for $\mathcal
  Z(\psi_0) \backslash \zbp$ and $\alpha_\psi$ for $\mathcal Z(\psi)
  \backslash \zbp$.  Since there are only finitely many choices of
  $\psi$ with $\mathcal Z(\psi) \subset \mathcal Z(\psi_0)$, the
  measure $\nu$ is finitely supported.

  Consider the coefficient of $z^{N-1}{\overline{w}}^{N-1}$
  in~\eqref{eq:20}.  On the left side, it is equal to
  $|S_{N-1}(\overline{\tilde\alpha})|^2 = \prod_{{\tilde\alpha}_j \in
    \mathcal Z(\psi_0)\backslash \zbp} |\alpha_j|^2 \neq 0$.  On the
  other hand, since $\deg\psi < N-1$, the coefficient on the right
  side must be $0$, giving a contradiction.  Thus $\mathcal C$ cannot
  have been a set of test functions for the algebra $\mathcal A_{B'}$,
  and so $\Psi_{B'}$ is a minimal set of test functions.
\end{proof}

The minimality of the set of test functions for $\mathcal A_B^0$
follows, with some variation, the same sort of reasoning as in the
proof of the last theorem.

\begin{theorem}
  \label{thm:minimal_set_of_test_fns-A_B0}
  The set $\Psi_{B'}^0$ is a minimal set of test functions for the
  algebra $\mathcal A_{B'}^0$.
\end{theorem}

\begin{proof}
  By Theorem~\ref{thm:algebras-isom-isom}, the case where $B'$ has
  $N=2$ zeros is covered by the last theorem, so assume from now on
  the $N > 2$.  The elements ${B'}$, $z{B'}$ and the set $\{{B'}^{N-1}
  m_\alpha : \alpha\in \hat{\mathbb D}\}$ form distinct components of
  $\Psi_{B'}^0$.  Hence removing an open subset of $\Psi_{B'}^0$
  amounts to removing ${B'}$, $z{B'}$ or an open subset of
  $\{{B'}^{N-1} m_\alpha\}$, or a union of such sets.  If the set
  removed contains a subset of $\{{B'}^{N-1} m_\alpha\}$, it is
  assumed that $\psi_0 = {B'}^{N-1} m_\alpha$ has been chosen with
  $\alpha \notin \zbp$.

  As before,
  \begin{equation}
    \label{eq:21}
    1 - \psi_0(z)\psi_0(w)^* = \int_{\mathcal C} \sum_{\ell = 1}^{L}
    h_{\psi,\ell} (z) (1-\psi(z)\psi(w)^*) h_{\psi,\ell} (w)^* \,
    d\nu(\psi),
  \end{equation}
  where now $L = N(N-1)+2$.  Denote zeros of $\psi_0$ by
  ${\tilde\alpha}_j$ and those of $\psi$ by $\alpha_j$.  For
  $n=0,\dots,N(N-1)$, there are constants $c_{nj}$ such that
  \begin{equation}
    \label{eq:22}
    h_{\psi,\ell} k_{\alpha_n} = \sum_{j=0}^{N(N-1)} c_{nj}\,
    k_{\tilde{\alpha}_{j}}.
  \end{equation}
  If $\psi_0 = B'$ or $zB'$, it is understood that the remaining
  $\alpha_j$s are $\infty$, and for these, $k_{\alpha_j} = 0$.  When
  $n=0$, \eqref{eq:22} gives
  \begin{equation}
    \label{eq:23}
    h_{\psi,\ell} = \sum_{j=0}^{N(N-1)} c_{0j}\,
    k_{\tilde{\alpha}_{j}}.
  \end{equation}
  Once again, all kernels and $h_{\psi,\ell}$ extend meromorphically
  to the Riemann sphere.  Substituting this back into \eqref{eq:21},
  for $n>1$,
  \begin{equation*}
    k_{\alpha_n} \sum_{j=0}^{N(N-1)} c_{0j}\, k_{\tilde{\alpha}_{j}} =
    \sum_{j=0}^{N(N-1)} c_{nj}\, k_{\tilde{\alpha}_{j}}.
  \end{equation*}

  Arguing as in the proof of
  Theorem~\ref{thm:minimal_set_of_test_fns}, $\mathcal Z(\psi)$ is a
  proper subset of $\mathcal Z(\psi_0)$.  In particular, $\psi_0 = B'$
  is immediately ruled out.  If $\psi_0 = zB'$,
  \begin{equation*}
    1-zB'(z)B'(w)^*w^* = \sum_\ell h_{B',\ell}(z)(1-{B'}(z){B'}(w)^*)
    h_{B',\ell}(w)^*.
  \end{equation*}
  The set $\zbp$ consists of $\{\alpha_0,\dots\alpha_m\}$, where
  $\alpha_j$ has multiplicity $t_j$.  Multiplying through by
  $\prod_j(1-\overline{\alpha_j}z)^{t_j}$ in \eqref{eq:23},
  \begin{equation*}
    g_{B',\ell}(z) = h_{B',\ell}(z)
    \prod_j(1-\overline{\alpha_j}z)^{t_j}
  \end{equation*}
  is a polynomial of degree less than $\deg B'$, and
  \begin{equation*}
    \begin{split}
      &\prod_j(1-\overline{\alpha_j}z)^{t_j}
      \prod_j(1-\overline{\alpha_j}w)^{*\,t_j} - z\prod_j (\alpha_j -
      z)^{t_j} (\alpha_j - w)^{*\,t_j}w^* -
      g_{B',\ell}(z)g_{B',\ell}(w)^* \\
      = & g_{B',\ell}(z){B'}(z){B'}(w)^*g_{B',\ell}(w)^*.
    \end{split}
  \end{equation*}
  If $g_{B',\ell} \neq 0$, there is a $w \in \mathbb D$ such that
  ${B'}(w)^*g_{B',\ell}(w)^* \neq 0$.  Thus $g_{B',\ell}(z){B'}(z)$ is
  a polynomial, and so the zeros of $g_{B',\ell}$ must cancel the
  poles of ${B'}(z)$.  But since $\deg g_{B',\ell} < \deg B'$, which is
  impossible.  Therefore $\psi_0 \neq zB'$.

  Now assume that $\psi_0 = {B'}^{N-1} m_\alpha$, where $0 \neq \alpha
  \notin \zbp$.  Since $|\mathcal Z(\psi)| < |\mathcal Z(\psi_0)|$,
  $\psi = {B'}$, $z{B'}$ or ${B'}^{N-1}$.  Since
  \begin{equation*}
    \begin{split}
      & 1-{B'}^{N-1}(z){B'}^{* N-1}(w) \\
      =\, & (1-{B'}(z){B'}(w)^*) +
      {B'}(z)(1-{B'}(z){B'}(w)^*){B'}(w)^* + \cdots +
      {B'}(z)^{N-2}(1-{B'}(z){B'}(w)^*){B'}(w)^{*\,N-2} \\
      =\, & h(z)(1-{B'}(z){B'}(w)^*)h(w)^*,
    \end{split}
  \end{equation*}
  $h(z) = \begin{pmatrix} 1 & B'(z) & \cdots & B'(z)^{N-2}
  \end{pmatrix}$, without loss of generality, $\psi \in\{B',zB'\}$.
  Hence
  \begin{equation}
    \label{eq:24}
    \begin{split}
      & 1 - {B'}^{N-1}(z) m_\alpha(z)m_\alpha(w)^*{B'}^{N-1}(w^*) \\
      = & \sum_\ell \left[h_{{B'},\ell}(z)(1-{B'}(z){B'}(w)^*)
        h_{{B'},\ell}(w)^* + h_{z{B'},\ell}(z)(1-z{B'}(z){B'}(w)^*w^*)
        h_{z{B'},\ell}(w)^*\right].
    \end{split}
  \end{equation}
  where $h_{{B'},\ell}$ and $h_{z{B'},\ell}$ are as in \eqref{eq:22}.

  As before, write the distinct elements of $\zbp$ as
  $\{0=\alpha_0,\dots\alpha_m\}$, where $\alpha_j$ has multiplicity
  $t_j$.  Then $\mathcal Z(\psi_0)$ consists of
  $\{\alpha_0,\dots\alpha_m,\alpha\}$, where $\alpha_j$ has
  multiplicity $(N-1)t_j$ and $\alpha$ has multiplicity~$1$.  By pole
  counting, the coefficient of the kernel
  $k_{\alpha_j}^{((N-2)t_j+\ell -1)}$, $1 \leq \ell \leq t_j$ in
  \eqref{eq:23} is $0$.  Moreover, when $\alpha_j = 0$ and $\psi =
  z{B'}$, the coefficient of $k_{0}^{((N-2)t_0 - 1)}$ also equals $0$.
  Therefore,
  \begin{equation*}
    h_{\psi,\ell} = g_{\psi,\ell} (1-\overline{\alpha}z)
    \prod_j(1-\overline{\alpha_j}z)^{-(N-2)t_j},
  \end{equation*}
  where $g_{\psi,\ell} = \sum_s g_{\psi,\ell,s} z^s = c
  \prod_j(z-\beta_j)$ is a polynomial of degree at most $N(N-2)-1$
  when $\psi = {B'}$ and $N(N-2)-2$ when $\psi = z{B'}$.  Write $Z$
  for the set of zeros of ${B'}^{N-2}m_\alpha$, counting
  multiplicities, and $\beta_{{B'},\ell}$, $\beta_{z{B'},\ell}$ for
  the set of roots of $g_{{B'},\ell}$ and $g_{z{B'},\ell}$.
  Multiplying through in equation \eqref{eq:24} by
  $(1-\overline{\alpha}z)\prod_j(1-\overline{\alpha_j}z)^{(N-2)t_j}
  \prod_j(1-\overline{\alpha_j}w)^{*\,(N-2)t_j}(1-\overline{\alpha}w)^*$
  and using \eqref{eq:1},
  \begin{equation}
    \label{eq:25}
    \begin{split}
      &\sum_{m,n = 0}^{N(N-2)+1} \left[z^m\,{\overline{w}}^n
        S_m(\overline{Z})S_n(\overline{Z})^* - B'(z)B'(w)^*
        z^{N(N-2)+1-m}\,{\overline{w}}^{N(N-2)+1-n} S_m(Z)
        S_n(Z)^*\right] \\
      = & \sum_\ell \sum_{m,n = 0}^{N(N-2)} |g_{{B'},\ell,0}|^2
      \left[\vphantom{{\overline{w}}^{N(N-2)-n}} z^m\,{\overline{w}}^n
        S_{N(N-2)-m}(\beta_{{B'},\ell})
        S_{N(N-2)-n}(\beta_{{B'},\ell})
      \right. \\
      & \left. \qquad - B'(z)B'(w)^*
        z^{N(N-2)-m}\,{\overline{w}}^{N(N-2)-n} S_m(\beta_{{B'},\ell})
        S_n(\beta_{{B'},\ell})
      \right] \\
      & + \sum_\ell \sum_{m,n = 0}^{N(N-2)-1} |g_{z{B'},\ell,0}|^2
      \left[\vphantom{{\overline{w}}^{N(N-2)-n}} z^m\,{\overline{w}}^n
        S_{N(N-2)-1-m}(\beta_{z{B'},\ell})
        S_{N(N-2)-1-n}(\beta_{z{B'},\ell})
      \right. \\
      & \left. \quad - B'(z)B'(w)^*
        z^{N(N-2)-m}\,{\overline{w}}^{N(N-2)-n}
        S_m(\beta_{z{B'},\ell}) S_n(\beta_{z{B'},\ell})\right].
    \end{split}
  \end{equation}

  On the left, the coefficient of
  $B'(z)B'(w)^*z^{N(N-2)+1}\,{\overline{w}}^{N(N-2)+1}$ in
  \eqref{eq:25} is $S_0(Z)S_0(Z)^* = 1$, while on the right, the
  coefficient is $0$, yielding a contradiction.  Hence $\Psi_{B'}^0$
  is a minimal set of test functions.
\end{proof}

Recalling the discussion preceding Theorem~\ref{thm:algs_isom_isom},
the following corollary is seen to hold.

\begin{corollary}
  \label{cor:minimal-set-of-test-fns}
  The set
  \begin{equation*}
    \Psi_B = \{ c B_\alpha : \alpha \text{ an ordered } 2N-1
    \text{ tuple containing the elements of } \mathcal Z(B) \text{ and
    } c = \overline{B_\alpha(1)}\}
  \end{equation*}
  is a minimal set of test functions for $\mathcal A_B$, and the set
  \begin{equation*}
    \Psi_B^0 = \{B\}\cup \{z B\} \cup \{ c B m_\alpha :
    \alpha\in\hat{\mathbb D} \text{ and } c =
    \overline{B(1)m_\alpha(1)}\},
  \end{equation*}
  is a minimal set of test functions for $\mathcal A_B^0$.
\end{corollary}

This leads us to a refinement of the realization theorem,
Theorem~\ref{thm:realization-theorem}.

\begin{theorem}
  \label{thm:refined-realization}
  Let $\Psi = \Psi_{B}$ (respectively, $\Psi_{B}^0$) be the minimal
  set of test functions for the algebra $\mathcal A_{B}$
  (respectively, $\mathcal A_{B}^0$).  For $\varphi: \mathbb D\to
  \mathbb C$, the following are equivalent:
  \begin{enumerate}
  \item $\varphi \in \mathcal A_{B}$ (respectively, $\mathcal
    A_{B}^0$) with $\|\varphi\| \leq 1$;
  \item There is a positive measure $\mu$ from $\mathbb D\times
    \mathbb D$ to $C(\Psi_{B})^*$ (respectively, $C(\Psi_{B}^0)^*$)
    and $H_\psi \in H^2(\mathbb D)$ such that for all $z,w\in \mathbb
    D$,
    \begin{equation*}
      1 - \varphi(z)\varphi(w)^* = \int_{\Psi} 
      H_\psi(z) (1-\psi(z)\psi(w)^*) H_\psi(w)^* \, d\mu_{z,w}(\psi).
    \end{equation*}
  \end{enumerate}
\end{theorem}

\smallskip

Given a finite set $S \subset \mathbb D$, $n = |S|$, write $\mathcal
C_{1,S}$ for the set of matrices in $M_n(\mathbb C)$ of the form
\begin{equation*}
  {\left(\int_{\mathcal C} (1-\psi(z)\psi(w)^*) \,
      d\mu_{z,w}(\psi)\right)}_{z,w\in S},
\end{equation*}
$\mathcal C = \Psi_{B}$ (respectively, $\Psi_{B}^0$), and
$(\mu_{z,w})$ an $M_n(\mathbb C)$-valued positive Borel measure on
$\Psi_{B}$.  This set is a norm closed cone, contains all positive
matrices, and is also closed under conjugation (see~\cite{MR2389623}).
The realization theorem then can be restated as saying that $\varphi$
is in the unit ball of $\mathcal A_{B}$ (respectively, $\mathcal
A_{B}^0$) if and only if for all finite sets $S \subset \mathbb D$,
the matrix $(1 - \varphi(z)\varphi(w)^*)_{z,w\in S} \in \mathcal
C_{1,S}$.

As usual, there is also an Agler-Pick interpolation
theorem~\cite{MR3584680} (but see also~\cite{MR2666471},
\cite{MR2514385}, \cite{MR2899979} and \cite{MR2480640}).

\begin{theorem}
  \label{thm:AP-int}
  Let $\Psi_{B}$ (respectively, $\Psi_{B}^0$) be the minimal set of
  test functions for the algebra $\mathcal A_{B}$ (respectively,
  $\mathcal A_{B}^0$).  Let $F$, a finite subset of $\mathbb D$, $|F|
  = n$, and $\xi:F\to \mathbb D$ be given.
  \begin{enumerate}
  \item There exists $\varphi$ in $\mathcal A_{B}$ (respectively,
    $\mathcal A_{B}^0$) satisfying $\|\varphi\| \le 1$ and
    $\varphi|_F=\xi$;
  \item for each $k$ in $\mathcal K_{\Psi_B}$ (respectively, $\mathcal
    K_{\Psi_B^0}$), the kernel defined by
    \begin{equation*}
      F\times F \ni (z,w) \mapsto (1-\xi(z)\xi(w)^*)k(z,w)
    \end{equation*}
    is positive.
  \end{enumerate}
\end{theorem}

\section{Completely contractive representations and dilations}
\label{sec:compl-contr-repr}

As Arveson showed, there is an intimate connection between completely
contractive representations and dilations.  For the disk algebra
$A(\mathbb D)$ and the bidisk algebra $A(\mathbb D^2)$, the Sz.-Nagy
dilation theorem and And\^o's theorem tell us that any representation
of one of these algebras which sends the generators to contractions is
automatically completely contractive.

For a constrained algebra $\mathcal A_B$, there is a similar
characterization of those representations which are completely
contractive.  This was first observed by Broschinski~\cite{MR3208801}
for the Neil algebra $\mathcal A_{z^2}$.

\begin{theorem}
  \label{thm:dilation-theorem}
  A unital representation $\pi: \mathcal A_B \to \mathcal{B(H)}$,
  $\mathcal H$ a Hilbert space, is completely contractive if and only
  if there is a unitary operator $U$ acting on a Hilbert space
  $\mathcal K \supset \mathcal H$ such that for $j\in \mathbb N$,
  \begin{equation*}
    \pi(z^j B) = P_\mathcal H U^j B(U) |_\mathcal H.
  \end{equation*}
\end{theorem}

\begin{proof}
  Suppose that $\pi$ is a map of the given form.  By linearity, $\pi$
  extends to functions of the form $pB$, $p$ a polynomial.  By the
  spectral theorem for normal operators, the representation is
  bounded, and so extends to a representation of $\mathcal A_B$.
  Since $B$ is inner, the spectrum of $x(U) := B(U)$ and $y(U) :=
  UB(U)$ define normal operators with spectrum on the boundary of
  $\mathcal N_B$, and so by Arveson's theorem, $\pi$ is completely
  contractive.

  Conversely, if $\pi$ is completely contractive, it induces a
  completely positive map on the operator space $\mathcal A_B +
  \mathcal A_B^*$ by $\pi(f+g^*) = \pi(f)+ \pi(g)^*$.  An application
  of the Arveson extension theorem extends $\pi$ to a completely
  positive map on $C(\mathbb T)$.  The Stinespring dilation theorem
  then yields a dilation of this to a representation $\rho$ with the
  property that $\rho(z) = U$, which is unitary.
\end{proof}

By using Theorem~\ref{thm:algebras-isom-isom}, the same argument gives
a dilation theorem for the algebra $\mathcal A_B^0$.

\begin{theorem}
  \label{thm:dil-thm-for-AB0}
  A unital representation $\pi: \mathcal A_B^0 \to \mathcal{B(H)}$,
  $\mathcal H$ a Hilbert space, is completely contractive if and only
  if there is a unitary operator $U$ acting on a Hilbert space
  $\mathcal K \supset \mathcal H$ such that for $1 \leq i \leq N-2$
  and $1 \leq j \leq i$, and for $i=N-1$ and $j\in \mathbb N$,
  \begin{equation*}
    \pi(z^j B^i) = P_\mathcal H U^j B(U)^i |_\mathcal H.
  \end{equation*}
\end{theorem}

As in~\cite{MR3584680}, it happens that even though there is a
contraction $T := P_\mathcal H U |_\mathcal H$, for neither algebra is
it necessarily the case that $\pi(B) = B(T)$ and $\pi(zB) = TB(T)$.

\begin{proposition}
  \label{prop:no-single-gen-example}
  For both $\mathcal A_B$ and $\mathcal A_B^0$, there is a completely
  contractive representation $\pi$ in $\mathcal{B(H)}$ for which there
  is no operator $T \in \mathcal{B(H)}$ such that $\pi(B) = B(T)$ and
  $\pi(zB) = TB(T)$.
\end{proposition}

\begin{proof}
  Consider $\mathcal A_B$ to begin with.  Let
  $\alpha_0,\dots,\alpha_m$ be the zeros of $B$ with multiplicities
  $t_0,\dots,t_m$, respectively, and recall the functions
  $g_1,\dots,g_{N-1}$ defined in terms of the kernel functions
  $k^{(i)}_{\alpha_j}$ (given in~\eqref{eq:10}) in the paragraphs
  preceding Theorem~\ref{thm:extr-pt-char}.  By definition,
  $k^{(i)}_{\alpha_j}$ is divisible by $z^i$ (and no higher power of
  $z$) and a simple calculation shows that likewise, the functions
  $k^{(0)}_{\alpha_\ell} - k^{(0)}_{\alpha_j}$, $j\neq \ell$ are
  divisible by $z$ but no higher power of $z$.  Each $g_j$ is in
  $H^2(\mathbb D)$, the functions in $L^2(\mathbb T)$ (with normalized
  Lebesgue measure) where the coefficients of $z^j$ are zero when
  $j<0$.

  Define $\mathcal H \subset H^2(\mathbb D)$ to be the orthogonal
  complement of the span of $g$, where either $g = k^{(1)}_{\alpha_j}$
  for some $j$ or $g = k^{(0)}_{\alpha_\ell} - k^{(0)}_{\alpha_j}$.
  The degree of $B$ is at least $2$, so there is always one such $g$
  in the set of complex annihilators of $\mathcal A_B$.  Since
  $\mathrm{ran}\, B$ is orthogonal to the span of $g$, $\mathcal H$ is
  invariant under multiplication by both $B$ and $zB$.

  Let $U$ be the bilateral shift on $L^2(\mathbb T)$, which is
  unitary.  Then $\mathcal H$ is invariant under both $B(U)$ and
  $UB(U)$.  Furthermore, $U^*h \in H^2(\mathbb D)$. and $z$ does not
  divide $U^*h$.  Since each $g_j$ is divisible by $z$, this implies
  that $U^*h$ is not in the annihilator of $\mathcal A_B$.

  Suppose that there exists $T \in \mathcal{B(H)}$ such that $\pi(B) =
  B(T) = B(U) |_\mathcal H$ and $\pi(zB) = TB(T) = UB(U) |_\mathcal
  H$.  As $B$ is inner, both $\pi(B)$ and $\pi(zB)$ are isometries.
  The quotient space $\hat{\mathcal H} = H^2(\mathbb D)/ \bigvee g$ is
  isometrically isomorphic to $\mathcal H$.  Let $q$ be the quotient
  map.  Since $\mathcal H$ is invariant under $U$, $T$ passes to a
  contraction operator $\hat T$ on the quotient space and ${\hat T}^j
  B(\hat T)$ are isometries, $j=0,1$.  Also, there is an isometry $V:
  \hat{\mathcal H} \to L^2(\mathbb T)$ such that $\hat\pi(z^jB) :=
  {\hat T}^j B(\hat T) = V^* U^jB(U) V$, $j\in \mathbb N \cup \{0\}$.
  Hence by Theorem~\ref{thm:dilation-theorem}, $\hat\pi$ defines a
  completely contractive representation of $\mathcal A_B$ into
  $\mathcal{B}(\hat{\mathcal H})$.

  Since $U (U^* g) = g$, $\hat T q(U^* g) = 0$.  As noted, the map
  $\hat T$ is isometric, and so it follows that $q(U^* g) = 0$.  But
  since $U^* g$ is not in the annihilator of $\mathcal A_B \supset
  \bigvee g$, $q(U^* g)$ cannot be $0$, giving a contradiction.

  The representation $\hat\pi$ of $\mathcal A_B$ constructed above
  restricts to a completely contractive representation of $\mathcal
  A_B^0$.  Since there is no operator $\hat T$ such that ${\hat T}^j
  B(\hat T)$, $j=0,1$, and these latter are in $\mathcal A_B^0$, the
  claim holds for $\mathcal A_B^0$ as well.
\end{proof}

Along the lines of the example due to Kaijser and Varopoulos on the
tridisk~\cite{MR0355642} (see also~\cite{MR3584680}), it will be shown
that for both $\mathcal A_B$ and $\mathcal A_B^0$ there are
representations sending the generators to contractions which are not
contractive.

\begin{theorem}
  \label{thm:KV-example}
  There is a non-contractive unital representation $\pi$ of $\mathcal
  A_B$ (respectively, $\mathcal A_B^0$) which maps the generators to
  contractions.
\end{theorem}

\begin{proof}
  Only the algebra $\mathcal A_B^0$ is considered, since the argument
  for $\mathcal A_B$ is identical.

  Let $\mathcal C_0 = \{B,zB\}$, the generators of $\mathcal A_B^0$.
  By Theorem~\ref{thm:refined-realization}, there is a finite set $F$
  and a function $\varphi$ in the unit ball of $\mathcal A_B^0$ such
  that $1 - \varphi(z)\varphi(w)^* = \int_{\Psi_B^0}
  (1-\psi(z)\psi(w)^*) \, d\mu_{z,w}(\psi)$, $z,w\in F$, but such that
  there is no finite positive Borel measure $(\mu_{xy}^0)_{x,y\in F}$
  with the property that $1 - \varphi(z)\varphi(w)^* = \int_{\mathcal
    C_0} (1-\psi(z)\psi(w)^*) \, d\mu^0_{z,w}(\psi)$ for all $z,w\in
  F$.  Consequently, there is a linear functional strictly separating
  the closed cone
  \begin{equation*}
    \left\{{\left(\int_{\mathcal C_0} (1-\psi(z)\psi(w)^*) \,
          d\mu^0_{z,w}(\psi)\right)}_{z,w\in F}: \mu^0_{z,w} \text{ a
        finite positive Borel measure}\right\}
  \end{equation*}
  from $(1 - \varphi(z)\varphi(w)^*)_{z,w\in F}$.  By a standard GNS
  construction, this results in a unital representation $\pi$ of
  $\mathcal A_B^0$ for which $\pi(B)$ and $\pi(zB)$ are contractions,
  yet $\pi(\varphi)$ is not contractive.
\end{proof}

Since $\mathcal A(\mathcal N_B)$ and $\mathcal A_B^0$ are (completely)
isometrically isomorphic, this of course means that there is also a
non-contractive unital representation of $\mathcal A(\mathcal N_B)$
which is contractive on generators.

\section{Contractive, but not completely contractive, representations
  of $\mathcal A_B$ and $\mathcal A_B^0$}
\label{sec:contractive-but-not-cc}

In this section, it is proved that for any $B$ with two or more zeros,
there exist contractive representations of $\mathcal A_B$ which are
not completely contractive.  Likewise, if $B$ has two or more zeros
all of the same multiplicity, there exist contractive representations
of $\mathcal A_B^0$ which are not completely contractive.  Indeed, in
all cases there is such a representation which is not $2$-contractive.
The minimal set of test functions is generically referred to as
$\Psi$.

Similarly to Section~\ref{sec:descr-test-funct}, for a finite set $S
\subset \mathbb D$, define $\mathcal C_{2,S}$ as the set of matrices
in $M_{2|S|}(\mathbb C)$ of the form
\begin{equation*}
  {\left(\int_\Psi (1-\psi(z)\psi(w)^*) \,
      d\mu_{z,w}(\psi)\right)}_{z,w\in S},
\end{equation*}
$(\mu_{z,w})$ a positive Borel measure on $\Psi$ with entries in
$M_2(\mathbb C)$.  As with $C_{1,S}$, this is a norm closed cone,
contains all positive matrices, and is closed under conjugation
(see~\cite{MR3584680}).

Given a finite set $S\subset \mathbb D$, let $I_S$ be the ideal of
functions in $\mathcal A = \mathcal A_B$ or $\mathcal A_B^0$ vanishing
on $S$.  The quotient map $q:\mathcal A \to \mathcal A/I_S$ is
completely contractive.  Assuming the set $S$ and a function $\Phi$ in
the unit ball of $M_2\otimes \mathcal A$ can be chosen so that $(I_2-
\Phi(z)\Phi(w)^*)_{z,w\in S} \notin \mathcal C_{2,S}$ a cone
separation argument and GNS construction implies that there is a
representation $\tau : \mathcal A/I_S \to \mathcal{B(H)}$ with the
property that $\pi = q \circ \tau$ is contractive but not
$2$-contractive (and hence not completely contractive)
(see~\cite{MR3584680},~Proposition~3.5).

Following~\cite{MR3584680}, let
\begin{equation}
  \label{eq:26}
  R(z) =
  \begin{pmatrix}
    m_{p_1} & 0 \\ 0 & 1
  \end{pmatrix}
  U
  \begin{pmatrix}
    1 & 0 \\ 0 & m_{p_2}
  \end{pmatrix},
\end{equation}
where $U$ a unitary matrix in $M_2(\mathbb C)$ with non-zero off
diagonal entries, concretely chosen as
\begin{equation*}
  U = \frac{1}{\sqrt{2}}
  \begin{pmatrix}
    1 & 1 \\ 1 & -1
  \end{pmatrix}.
\end{equation*}
Define $\Phi = B^n(z)R(z)$ with $n=1$ when $\Phi\in M_2\otimes
\mathcal A_B$ and $n=N-1$ when $\Phi\in M_2\otimes \mathcal A_B^0$.
In both cases, $\|\Phi\| \leq 1$.

From here on, $p_1,p_2 \notin \zb\cup \{0\}$ are taken to be distinct
points, and $S$ is a set of $2N^2-3N+5$ points in $\mathbb D$
containing $p_1$, $p_2$ and the zeros in $\zb$ (including repeated
roots).  In this case, $S' := S\backslash \{\mathcal
Z(B)\cup\{p_1,p_2\}\}$ consists of $2(N-1)^2+1$ distinct points, and
we assume that these are chosen so that any polynomial which is zero
on $S'$ has degree greater than $2(N-1)^2$.  Define
\begin{equation*}
  \Delta_{\Phi,S} = {\left(1-\Phi(z)\Phi(w)^*\right)}_{z,w\in S}.
\end{equation*}

There are several results from \cite{MR3584680} which will be needed
in what follows.  Some include small variations on what is found
there.  Where the proofs are essentially unaltered, they are left out.

Since the zeros of $B$ are $\{\alpha_j\}_0^m$ with multiplicities
$\{t_j\}_0^n$, $\sum t_j = N$, it follows that $B^{n-1}$ has the same
zeros, but with multiplicities $\{(n-1)t_j\}$ summing to $(n-1)N$.
There are then $(n-1)N$ linearly independent kernels
$\{k_{\alpha_0}^{(s)}: 0 \leq s \leq (n-1)t_0-1\} \cup \cdots \cup
\{k_{\alpha_m}^{(s)}: 0 \leq s \leq (n-1)t_m-1\}$.  As a shorthand, we
write $\{{\tilde k}_i\}_{i=1}^{(n-1)N}$ for these kernels.  When
$n=1$, this is taken to be the empty set.

\begin{lemma}[{\citep[Lemma 4.3]{MR3584680}}]
  \label{lem:DJM-4.3}
  There exist linearly independent vectors $v_1,v_2\in\mathbb C^2$
  along with $2(n-1)N +2 = 2n^2$ functions $a_j : S \to\mathbb C^2$ in
  the span of
  \begin{equation*}
    E = \{k_{p_1} v_1, {\tilde k}_1 v_1 , \dots , {\tilde k}_{(n-1)N}
    v_1 \} \cup \{k_{p_2} v_2, {\tilde k}_1 v_2 , \dots , {\tilde
      k}_{(n-1)N} v_2 \}
  \end{equation*}
  (with $E = \{k_{p_1} v_1, k_{p_2} v_2\}$ when $n=1$) such
  that
  \begin{equation*}
    \frac{I_2-B^{n-1}(z)R(z)R(w)^*B^{n-1}(w)^*}{1-zw^*} =
    \sum_{j=1}^{2n^2} a_j(z)a_j(w)^*.
  \end{equation*}
\end{lemma}

\smallskip

With the algebra $\mathcal A_B$ the number of terms will be $2$, while
for $\mathcal A_B^0$ it will be $2(N-1)^2$.

For $\zeta \in \mathbb D$,
\begin{equation*}
  k_\zeta(z) = \frac{\sqrt{1 - |\zeta|^2}}{1-z\overline{\zeta}},
\end{equation*}
denotes the normalized Szeg\H{o} kernel, with $k_\infty = 0$.  Then
for all $\zeta \in \hat{\mathbb D}$ (recall that $m_\infty = 1$),
\begin{equation*}
  \frac{1-m_\zeta(z)m_\zeta(w)^*}{1-zw^*} = k_\zeta(z)k_\zeta(w)^*.
\end{equation*}
More generally, if $G$ is a Blaschke product with zero set $\mathcal
Z(G) = \{\zeta_0 = \infty, \zeta_1, \dots, \zeta_\ell\}$ (including
multiplicities),
\begin{equation*}
  \begin{split}
    \frac{1-G(z)G(w)^*}{1-zw^*} &= \sum_{j=1}^\ell
    \left(\prod_{i=0}^{j-1} m_{\zeta_i}\right)
    \frac{1-m_{\zeta_j}(z)m_{\zeta_j}(w)^*}{1-zw^*}
    \left(\prod_{i=1}^{j-1} m_{\zeta_i}\right) \\
    &= \sum_{j=1}^\ell \left(\prod_{i=0}^{j-1} m_{\zeta_i}\right)
    k_{\zeta_j}(z)k_{\zeta_j}(w)^*
    \left(\prod_{i=1}^{j-1} m_{\zeta_i}\right) \\
    & = K_\zeta(z) K_\zeta(w)^*,
  \end{split}
\end{equation*}
where $K_\zeta = \begin{pmatrix} k_{\zeta_1} & k_{\zeta_2}m_{\zeta_1}
  & \cdots & k_{\zeta_\ell} \prod_{i=1}^{\ell-1} m_{\zeta_i}
\end{pmatrix}$.

Apply this to $B^{n-1} R_\lambda$, where $R_\lambda =
\prod_{j=1}^{N-n} m_{\lambda_j}$ with $\lambda =
(\lambda_j)_{j=0}^{N-n}\in {\hat{\mathbb D}}^{N-n}$, to obtain
\begin{equation}
  \label{eq:27}
  \frac{1-\psi_\lambda(z)\psi_\lambda(w)^*}{1-zw^*} =
  \frac{1-B(z)B(w)^*}{1-zw^*} + B(z) K_\lambda(z)K_\lambda(w)^*B(w)^*.
\end{equation}
When $n=1$ (for $\mathcal A_B$), $K_\lambda = \begin{pmatrix}
  k_{\lambda_1} & \cdots & k_{\lambda_{N-1}} \prod_{i=1}^{N-2}
  m_{\lambda_i} \end{pmatrix}$, where any term with $\lambda_i =
\infty$ is $0$, and when $n = N-1$ (for $\mathcal A_B^0$), $K_\lambda
= \begin{pmatrix} k_{\alpha_1} & \cdots & k_{\lambda} \prod_{i=1}^{N}
  m_{\lambda_i}^{N-2} \end{pmatrix}$, and only the last term involves
$\lambda \in\mathbb D$.

Suppose that $\Delta_{\Phi,S} \in \mathcal C_{2,S}$.  Applying
\eqref{eq:27} and Lemma~\ref{lem:DJM-4.3}, there exist linearly
independent vectors $v_1,v_2\in\mathbb C^2$ and functions $a_j: S \to
\mathbb C^2$ in the span of $E$ such that
\begin{equation}
  \label{eq:28}
  \begin{split}
    \frac{I_2-\Phi(z)\Phi(w)^*}{1-zw^*} &= \frac{1-B(z)B(w)^*}{1-zw^*}
    I_2 + B(z)\frac{I_2 - B^{n-1}(z)R(z)
      R(w)^*B^{n-1}(w)^*}{1-zw^*}B(w)^* \\
    & = \frac{1-B(z)B(w)^*}{1-zw^*}
    I_2 + B(z)\left( \sum_1^{2n^2} a_j(z)a_j(w)^* \right)B(w)^* \\
    &= \frac{1-B(z)B(w)^*}{1-zw^*} \mu_{z,w}(\Psi)  + B(z)B(w)^*
    \mu_{zw}(\{zB\}) \\
    & \qquad + B(z)B(w)^* \int_{\Psi\backslash \{B, zB\}}
    \frac{1-B^{n-1}(z)R_\lambda(z) R_\lambda(w)^*B^{n-1}(w)^*}{1 -
      zw^*}\, d\mu_{z,w}(\psi_\lambda) \\
    &= \frac{1-B(z)B(w)^*}{1-zw^*} \mu_{z,w}(\Psi)  + B(z)B(w)^*
    \mu_{zw}(\{zB\}) \\
    & \qquad + B(z)B(w)^* \int_{\Psi\backslash \{B, zB\}}
    K_\lambda(z)K_\lambda(w)^*\, d\mu_{z,w}(\psi_\lambda).
  \end{split}
\end{equation}

Define positive (ie, positive semidefinite) kernels
$A$, $D$ and $\tilde D$ on $S\times S$ by
\begin{equation*}
  \begin{split}
    A(z,w) &= \mu_{z,w}(\Psi) \\
    D(z,w) &= B(z)\left( \sum_1^{2n^2} a_j(z)a_j(w)^* \right)B(w)^* \\
    \tilde D(z,w) &= B(z) \left(\mu_{zw}(\{zB\}) +
      \int_{\Psi\backslash \{B, zB\}} K_\lambda(z)K_\lambda(w)^*\,
      d\mu_{z,w}(\psi_\lambda) \right) B(w)^*.
  \end{split}
\end{equation*}
Then
\begin{equation*}
  D(z,w) - \tilde D(z,w) = \frac{1-B(z)B(w)^*}{1-zw^*}
  \left(A(z,w) - I_2 \right).
\end{equation*}

\begin{lemma}[See also {\citep[Lemma 5.2]{MR3584680}}]
  \label{lem:DJM-5.2}
  Assume that $\Delta_{\Phi,S}\in C_{2,S}$.  With the above notation,
  \begin{enumerate}[(i)]
  \item The $M_2(\mathbb C)$ valued kernel $A - [I] :=
    (A(z,w)-I_2)_{z,w\in S}$ is positive;
  \item The $M_2(\mathbb C)$ valued kernel $D - \tilde D$ is positive
    with rank at most $2(n-1)N +2$;
  \item The range of $\tilde{D}$ lies in the range of $D$, which is in
    $E_B = BE$; and
  \item For $z,w\in S' = S\backslash \{\mathcal
    Z(B)\cup\{p_1,p_2\}\}$, there are at most $2n^2$ functions $r_j:
    S'\to \mathbb C^2$ such that for $i=1,\dots N$, $r_j m_{\alpha_0}
    \cdots m_{\alpha_{i-1}} k_{\alpha_i} \in E$ ($m_{\alpha_0} \cdots
    m_{\alpha_{i-1}} =1$ if $i=0$) and
    \begin{equation*}
      A(z,w) = I_2 + \sum_j r_j(z)r_j(w)^*.
    \end{equation*}
    Furthermore, if $r_j(z) \neq 0$ for some $z\in S'$ then there are
    at most $2n^2 -1$ points in $S'$ where $r_j$ is zero.
  \end{enumerate}
\end{lemma}

\begin{proof}
  Recall that $\mathcal Z = \mathcal Z(B) \subseteq \mathcal Z(\psi)$
  for all $\psi\in \Psi$.  Hence for $\alpha \in \mathcal Z(B)$ and
  $w\in S$,
  \begin{equation*}
    I_2 = I_2 - \Phi(\alpha)\Phi(w)^* = \int_{\Psi}
    (1-\psi(\alpha)\psi(w)^*)\,d\mu_{\alpha,w}(\psi_\lambda) =
    \int_{\Psi} d\mu_{\alpha,w}(\psi_\lambda) = A(\alpha,w).
  \end{equation*}
  Fix $\alpha$ and factor $(A(z,w))_{z,w \neq \alpha} = LL^*$.  By
  positivity of $A$, there is a contraction $G$ such that $H$, the
  column matrix of $2N-3$ copies of $I_2$, can be factored as $H =
  LG$.  Hence $LL^* \geq LGG^*L^* = HH^*$, a $(2N-3)\times (2N-3)$
  matrix with all entries equal to $I_2$.  Hence
  \begin{equation*}
    A \geq
    \begin{pmatrix}
      HH^* & H \\ H^* & I_2
    \end{pmatrix}
    = [I_2].
  \end{equation*}
  This takes care of~(i).

  The kernel $\left(\frac{1-B(z)B(w)^*}{1-zw^*}\right) = K_{\mathcal
    Z}(z)K_{\mathcal Z}(w)^* \geq 0$.  Since the Schur product of
  positive matrices is positive,
  \begin{equation*}
    D - \tilde D = \left(\frac{1-B(z) B(w)^*}{1-zw^*}\right)
    \left(A - [I_2]\right) \geq 0.
  \end{equation*}
  Thus $\ran \tilde D \subset \ran D \subset BE$,
  proving~(ii) and~(iii).

  Now turn to (iv).  By the proof of~(i), the rank-nullity theorem,
  and since $\left(\frac{1-B(z)B(w)^*}{1-zw^*}\right)_{z,w\in S'} >
  0$, the rank of $\left(A(z,w) - I_2\right)_{z,w\in S'}$ is at most
  $2n^2$, and so $\left(A - I_2\right)(z,w) = \sum_1^{2n^2}
  r_j(z)r_j(w)^*$, where $r_j : S' \to \mathbb C^2$.  Thus
  \begin{equation}
    \label{eq:29}
    \left(K_{\mathcal Z}(z)\sum_1^{2n^2} r_j(z)r_j(w)^* K_{\mathcal
      Z}(w)^*\right) \leq (D(z,w)) = \left(B(z)\left( \sum_1^{2n^2}
      a_j(z)a_j(w)^* \right)B(w)^*\right)
  \end{equation}

  The left side of \eqref{eq:29} is the sum of positive
  matrices of the form
  \begin{equation*}
    \left(\left(r_j(z)m_{\alpha_0}(z)\cdots
      m_{\alpha_{i-1}}(z)k_{\alpha_i}(z)
      k_{\alpha_i}(w)^*m_{\alpha_{i-1}}(w)^* \cdots
      m_{\alpha_1}(w)^*r_j(w)^* \right)\right),
  \end{equation*}
  $1 \leq j \leq 2n^2$, $1 \leq i \leq N$, where $m_{\alpha_1}\cdots
  m_{\alpha_{i-1}} = 1$ if $i = 0$.  Consequently,
  \begin{equation}
    \label{eq:30}
    r_j m_{\alpha_1} \cdots m_{\alpha_{i-1}} k_{\alpha_i} \in BE.
  \end{equation}
  It is worth noting for later use that the order of the elements in
  $\mathcal Z$ does not effect the $r_j$s.

  Write
  \begin{equation}
    \label{eq:31}
    r_j m_{\alpha_0} \cdots m_{\alpha_{i-1}} k_{\alpha_i} = w_{1ji}
    v_1 + w_{2ji} v_2,
  \end{equation}
  where
  \begin{equation}
    \label{eq:32}
    w_{1ji} = B\left(c_{ji0} k_{p_1} + \sum_\ell c_{ji\ell} {\tilde
        k}_\ell \right),
  \end{equation}
  with a similar expression for $w_{2ji}$.  (The sums are absent if
  $N=2$, and the notation introduced just before
  Lemma~\ref{lem:DJM-4.3} is used.)  If $r_j(z) \neq 0$ for some $z
  \in S'$, then either $w_{1ji}(z) \neq 0$ or $w_{2ji}(z) \neq 0$.
  Assume it is $w_{1ji}$, since the other case is identical.  Clearing
  the denominators of the term in parentheses in~\eqref{eq:32}, we
  have a non-trivial polynomial of degree less than $2n^2-1$.  By the
  assumption we have made regarding $S'$ and since $B$ is non-zero on
  $S'$, if this polynomial is $0$ at more than $2n^2-1$ points, then
  it is identically zero on $\mathbb D$, and consequently, $w_{1ji} =
  0$ there, giving a contradiction.  The functions $m_{\alpha_i}$ and
  $k_{\alpha_i}$ are nonzero on $S'$, so $r_j(z) = 0$ on a set of
  cardinality at most $2n^2-1$ if $r_j \neq 0$.
\end{proof}

\begin{lemma}[{See \cite[Lemma~5.3]{MR3584680}}]
  \label{lem:DJM-5.3}
  Assume $n=1$ (corresponding to the algebra $\mathcal A_B$), or $n =
  N-1 > 1$ (corresponding to the algebra $\mathcal A_B^0$), in which
  case it is assumed that $B$ has more than one distinct zero and all
  zeros of $B$ have the same multiplicity.  If $\Delta_{\Phi,S}\in
  C_{2,S}$, then $A = [I_2]$.
\end{lemma}

\begin{proof}
  If $n=1$ (corresponding to the algebra $\mathcal A_B$), then as
  noted following \eqref{eq:32}, $w_{1ji} = c_{ji0} k_{p_1}B$.  By
  Lemma~\ref{lem:DJM-5.2}, if $r_j \neq 0$, $r_j(z) \neq 0$ on a
  subset of $S'$ of cardinality at least $2(N-1)^2+1$.  Since $r_j
  k_{\alpha_1} = c_{j10} k_{p_1}B = c_{ji0} k_{p_1}k_{\alpha_1}B$, and
  $B(z) \neq 0$ if $z\in S'$, $c_{j10} k_{p_1} = c_{ji0}
  k_{p_1}k_{\alpha_1}$ on $S'$ and hence meromorphically on
  $\hat{\mathbb C}$.  Since $p_1\neq \alpha_1$, pole counting gives
  $c_{j10} = c_{ji0} = 0$.  If $\alpha_1$ is a zero of $B$ of
  multiplicity $t_1 >1$, a similar argument with $r_j
  m_{\alpha_1}^\ell k_{\alpha_1}$, with $1 \leq \ell \leq t_1$ can
  instead be used.  Consequently $r_j = 0$ and so $A = [I_2]$.

  Next turn to $\mathcal A_B^0$.  Assume $n = N-1 > 1$, that $B$ has
  more than one distinct zero and all zeros of $B$ have the same
  multiplicity.  Fix $j$ and write $r = r_j$ and
  $\{\alpha_i\}_{i=0}^m$ for the zeros of $B$, with multiplicities
  $t$.  Applying a M\"obius map if necessary, there is no loss in
  generality in assuming that there is an $i$ such that $\alpha_i =
  0$.

  As was pointed out after \eqref{eq:30}, the order chosen for the
  zeros does not effect $r$.  So given a permutation $\sigma$ of the
  numbers $\{0,\dots,m\}$, it follows that
  \eqref{eq:31} and \eqref{eq:32},
  \begin{equation*}
    r m_{\alpha_{\sigma^{-1}(0)}} \cdots
    m_{\alpha_{\sigma^{-1}(i-1)}} k_{\alpha_{\sigma^{-1}(j)}} =
    w^\sigma_{1j} v_1 + w^\sigma_{2j} v_2,
  \end{equation*}
  where
  \begin{equation*}
    w^\sigma_{1j} = B\left(c^\sigma_{j0} k_{p_1} + \sum_{\ell=0}^m
      \sum_{s=0}^{(n-1)t - 1} c^\sigma_{j\ell s} k_\ell^{(s)}\right),
  \end{equation*}
  with a similar expression for $w^\sigma_{2j}$.

  Suppose the $i$th M\"obius map $m_i$ has been applied so that
  $\alpha_i = 0$.  Choose a permutation $\sigma$ so that
  $\sigma^{-1}(i) = 0$ and for some $1 \leq i'\leq m$,
  $\sigma^{-1}(i') = 1$.  Write $p = m_{\alpha_0}(p_1)$, $\beta_s =
  m_{\alpha_0}(\alpha_s)$, $\kappa_p = k_p$, $\kappa_s^{(\ell)} =
  k_{\beta_s}^{(\ell)}$, $m_s = m_{m_{\alpha_0}(\alpha_s)}$, and $B'$
  for $B$ after the application of $m_i$.  Then $w^\sigma_{1j}$
  becomes
  \begin{equation}
    \label{eq:33}
    W^\sigma_{1j} = B'\left(c^\sigma_{j0} \kappa_p + \sum_{\ell=0}^m
      \sum_{s=0}^{(n-1)t - 1} c^\sigma_{j\ell s} \kappa_\ell^{(s)}
    \right),
  \end{equation}
  and
  \begin{equation}
    \label{eq:34}
    W^\sigma_{1j} = W^\sigma_{10}
    m_{\sigma^{-1}(1)}\cdots m_{\sigma^{-1}(j-1)} k_{\sigma^{-1}(j)}.
  \end{equation}
  As long as $\sigma^{-1}(i) = 0$, the coefficients in $W^\sigma_{10}$
  do not otherwise depend on $\sigma$.

  By \eqref{eq:34} with $j=1$ and $\sigma^{-1}(1) = i'$, unless the
  coefficient $c^\sigma_{j\ell ,(n-1)t - 1} = 0$, the right side has a
  pole of higher order than the left at $1/{\overline{\alpha}}_{i'}$.
  Allowing $i'$ to run over all possible choices, the result is that
  \begin{equation*}
    W^\sigma_{11} = B'\left(c^\sigma_{10} \kappa_p + 
    c^\sigma_{10 ,(n-1)t - 1} \kappa_0^{((n-1)t - 1)} + 
    \sum_{\ell=0}^m \sum_{s=0}^{(n-1)t - 2} 
    c^\sigma_{1\ell s} \kappa_\ell^{(s)}\right).
  \end{equation*}

  Under the inverse M\"obius map, this becomes
  \begin{equation*}
    r k_{\alpha_i} = B \left( c_{i0} k_{p_1} + c_{ii ,(n-1)t - 1}
    k_{\alpha_i}^{((n-1)t - 1)} + \sum_{\ell=0}^m
    \sum_{s=0}^{(n-1)t - 2} c_{i\ell s} k_\ell^{(s)} \right).
  \end{equation*}
  Multiply this equation by $k_{\alpha_{i'}}$, $i' \neq i$, and
  similarly, multiply the equation for $rk_{\alpha_{i'}}$ by
  $k_{\alpha_i}$.  Take the difference.  Then since $B$ is non-zero on
  $S'$, the term
  \begin{equation*}
    c_{ii ,(n-1)t - 1} k_{\alpha_i}^{((n-1)t -
      1)}k_{\alpha_{i'}} - c_{i'i' ,(n-1)t - 1}
    k_{\alpha_{i'}}^{((n-1)t - 1)}k_{\alpha_i} = 0,
  \end{equation*}
  as it otherwise is the only term with a pole of order $(n-1)t -1$ at
  $\infty$.  Thus
  \begin{equation*}
    c_{ii ,(n-1)t - 1} k_{\alpha_i}^{((n-1)t -
      2)} = c_{i'i' ,(n-1)t - 1}
    k_{\alpha_{i'}}^{((n-1)t - 2)},
  \end{equation*}
  and since $i'$ was arbitrary, linear independence of the kernels
  then implies that $c_{ii ,(n-1)t - 1} = 0$ for all $i$.  So for all
  $\sigma$, \eqref{eq:33} becomes
  \begin{equation*}
    W^\sigma_{1j} = B'\left( c^\sigma_{j0} \kappa_p + \sum_{\ell=0}^m
      \sum_{s=0}^{(n-1)t - 2} c^\sigma_{j\ell s} \kappa_\ell^{(s)}
    \right), \qquad j = 0,1.
  \end{equation*}

  Now repeating the same argument sufficiently many times, this
  last equation reduces to
  \begin{equation*}
    W^\sigma_{1j} = B'\left( c^\sigma_{j0} \kappa_p + \sum_{\ell=0}^m
      \kappa_\ell \right), \qquad j = 0,1.
  \end{equation*}
  Hence 
  \begin{equation*}
    r k_{\alpha_i} = B\left( c_{i0} k_{p_1} + \sum_{\ell=0}^m c_{i\ell}
      k_{\alpha_\ell}\right) , \qquad i = 1,\dots m.
  \end{equation*}
  If $r(z) = 0$ for $m+1$ or more choices of $z\in S'$, then since
  $B(z) \neq 0$ on $S'$ and by linear independence of the kernels on
  the right, all coefficients are zero, meaning that $r = 0$.

  So assume this is not the case.  For $i' \neq i$, $k_{\alpha_{i'}} r
  k_{\alpha_i} = k_{\alpha_i} r k_{\alpha_{i'}}$.  Since $S'$ is
  sufficiently large, this must hold for all $z\in \mathbb D$, and
  then extends meromorphically to $\hat{\mathbb C}$.  But then
  $k_{\alpha_{i'}} r k_{\alpha_i}$ has a pole at
  $1/{\overline{\alpha}}_{i'}$ of order one larger than $k_{\alpha_i}
  r k_{\alpha_{i'}}$ does unless $c_{ii'} = 0$.  Since $i'$ was
  arbitrary, $c_{ii'} = 0$ for all $i' \neq i$.  Hence for all $i$
  \begin{equation*}
     r k_{\alpha_i} = B\left(c_{i0} k_{p_1} + c_{ii}
       k_{\alpha_i}\right),
  \end{equation*}
  and so
  \begin{equation*}
    r(1-{\overline{p}_1} z) = B\left(c_{i0} (1-{\overline{\alpha}_i}
      z) + c_{ii} (1-{\overline{p}_1} z)\right), \qquad i = 1,\dots ,
    m.
  \end{equation*}
  It follows from the assumption $m > 1$ that
  \begin{equation*}
    c_{i0} (1-{\overline{\alpha}_i} z) + c_{ii} (1-{\overline{p}_1} z)
    = c_{i'0} (1-{\overline{\alpha}_{i'}} z) + c_{i'i'}
    (1-{\overline{p}_1} z), \qquad \alpha_{i'} \neq \alpha_i.
  \end{equation*}
  Evaluating at three non-collinear points in $S \backslash \zb$ gives
  $c_{i0} = c_{ii} = 0$ for all $i$, which yields a contradiction.  As
  a consequence, $r$ must be $0$.  The proof of the lemma now follows
  for $\mathcal A_B^0$ as well.
\end{proof}

For $\mathcal A_B^0$, the last lemma covers, among other things, the
setting where all zeros of $B$ are distinct.  In fact the result holds
more broadly, such as when $B$ has three zeros, two of which are the
same and the third distinct, though the proof becomes more
complicated.  Despite our best efforts, it appears that the case when
$B$ has $N >2$ identical zeros cannot be done in this way, at least
with the choice made of the function $\Phi$.  We do not have a
succinct characterization of all the possible choices of roots of $B$
for which the lemma holds with this choice of $\Phi$.

With minor notational changes, the proof of the following is
essentially that of \citep[Lemma 5.5]{MR3584680}.  It recalls
\eqref{eq:28} and uses the fact, proved in Lemma~4.2 of
\cite{MR3584680}, that for positive $M_2(\mathbb C)$ valued measures
$\mu_{z,w}$ with the property that $\mu_{z,w}(\Psi) = I_2$ for all
$z,w\in S$, there is a measure $\mu$ independent of $z$ and $w$ such
that $\mu_{z,w} = \mu$ for all $z,w$.

\begin{lemma}[{\citep[Lemma 5.5]{MR3584680}}]
  \label{lem:DJM-5.5}
  For the algebra $\mathcal A_B$, if $\Delta_{\Phi,S}\in C_{2,S}$,
  then there exists an $M_2(\mathbb C)$ valued measure $\mu$ on $\Psi$
  such that $\mu(\Psi) = I_2$ and for all $z,w\in S\backslash \mathcal
  Z(B)$,
  \begin{equation*}
    \frac{I_2 - R(z)R(w)^*}{1-zw^*} = \int_{\Psi\backslash \{B\}}
    K_\lambda(z)K_\lambda(w)^*\, d\mu(\psi_\lambda).
  \end{equation*}
\end{lemma}

There is a similar lemma for $\mathcal A_B^0$.

\begin{lemma}
  \label{lem:DJM-5.5-2}
  Assume $B$ has two or more distinct zeros and all zeros of $B$ have
  the same multiplicity.  For the algebra $\mathcal A_B^0$, if
  $\Delta_{\Phi,S}\in C_{2,S}$, then there exists an $M_2(\mathbb C)$
  valued measure $\mu$ on $\Psi$ such that $\mu(\Psi) = I_2$ and for
  all $z,w\in S\backslash \mathcal Z(B)$,
  \begin{equation*}
    \frac{I_2 - B^{N-2}(z)R(z)R(w)^*B^{N-2}(w)^*}{1-zw^*} =
    \mu(\{zB\}) + \int_{\Psi\backslash \{B,Bz\}} K_\lambda(z)
    K_\lambda(w)^*\, d\mu(\psi_\lambda).
  \end{equation*}
\end{lemma}

Here $\lambda$ lies in $\mathbb D\backslash\zb$.

For $\nu$ a $2\times 2$ matrix valued measure and $\gamma\in\mathbb
C^2$, define the scalar measure $\nu_\gamma(\omega)= \gamma^*
\nu(\omega)\gamma$.  In case $\nu$ is a positive measure, $\nu_\gamma$
is also positive.

\begin{lemma}[See also {\citep[Lemma 4.5]{MR3584680}}]
  \label{lem:DJM-4.5}
  Suppose $\nu$ is an $M_2(\mathbb C)$-valued positive measure on
  $\Psi\backslash \{B\}$.  For each $\gamma\in\mathbb C^2$ the measure
  $\nu_\gamma$ is a nonnegative linear combination of at most $n$
  point masses if and only if there exist $($possibly not distinct$)$
  points $\zeta_1,\dots,\zeta_n \in {\hat{\mathbb D}}^{N-1}\backslash
  \{\infty^{N-1}\}$ and positive semidefinite matrices $Q_1,\dots ,
  Q_n$ such that
  \begin{equation*}
    \nu = \sum_{j=1}^n \delta_{\zeta_j} Q_j,
  \end{equation*}
  where for each $j$, $\delta_{\zeta_j}$ is a scalar unit point
  measure on $\Psi\backslash \{B\}$ supported at $\psi_{\zeta_j}$.
\end{lemma}

\begin{proof}
  One direction is obvious, so for the converse, assume that every
  $\nu_\gamma$ is a nonnegative linear combination of at most $n$
  point masses.  Let $\nu$ be a $M_2(\mathbb C)$-valued measure on
  $\Psi\backslash \{B\}$, written as $(\nu_{ij})$ with respect to the
  standard basis $\{e_1,e_2\}$.  Obviously, for $i=1,2$, $\nu_{ii} =
  e_i^* \nu e_i$ is a positive measure and $\nu_{ji} = \nu_{ij}^*$.
  If $\nu_{ii}(\Omega) = 0$ for a Borel subset $\Omega \subset
  \Psi\backslash \{B\}$, then $\nu_{ij}(\Omega) = 0$.  Hence
  $\nu_{ij}$ is absolutely continuous with respect to both $\nu_{11}$
  and $\nu_{22}$, and has its support contained in the intersection of
  the supports of these measures.

  Let $n_{ij} = |\mathrm{supp}\,\nu_{ij}|$.  For $\gamma =
  {\begin{pmatrix} \gamma_1 & \gamma_2 \end{pmatrix}}^t$,
  \begin{equation*}
    \nu_\gamma = |\gamma_1|^2 \sum_{\ell = 1}^{n_{11}} c^{1,1}_\ell
    \delta_{\eta_{1,1},\ell} + |\gamma_2|^2 \sum_{\ell = 1}^{n_{22}}
    c^{2,2}_\ell \delta_{\eta_{2,2},\ell} + 2\sum_{\ell = 1}^{n_{12}}
    \mathrm{Re}\, (\gamma_1\gamma_2^* c^{1,2}_\ell)\,
    \delta_{\eta_{1,2},\ell}.
  \end{equation*}
  Assuming $\gamma_1, \gamma_2 = |\gamma_2|e^{i\theta}$ are
  nonzero, the set $C = \left\{\gamma_1 |\gamma_2|e^{-i\theta}
    c^{1,2}_\ell: \theta \in [0,2\pi)\right\}$ is a circle, and there
  are at most two values of $\theta$ where $2\mathrm{Re}\,
  (\gamma_1\gamma_2^* c^{1,2}_\ell) = -|\gamma_1|^2
  c^{1,1}_{\ell_1} - |\gamma_2|^2 c^{2,2}_{\ell_2}$, with
  $\delta_{\eta_{1,2},\ell} = \delta_{\eta_{1,1},\ell_1} =
  \delta_{\eta_{2,2},\ell_2}$.  Ranging over all $\ell$, there are at
  most a finite number of such $\theta$.  Choosing $\theta$ avoiding
  these points, it follows that $\mathrm{supp}\,\nu_\gamma =
  \mathrm{supp}\,\nu_{11} \cup \mathrm{supp}\,\nu_{22}$.  By
  assumption, at most $n$ of these points can be distinct, and hence
  $\nu$ has the form claimed.
\end{proof}

\begin{lemma}[See also {\citep[Lemma 5.6]{MR3584680}}]
  \label{lem:DJM-5.6}
  If $\Delta_{\Phi,S}\in C_{2,S}$ and the algebra is $\mathcal A_B$,
  then the measure $\mu$ has the form $\mu = \delta_1 P_1 + \delta_2
  P_2 + \delta_{12} P_{12} + \delta_{\infty} P_{\infty}$, where each
  $P_*$ is a $2\times 2$ positive matrix, $P_\infty + P_1 + P_2 +
  P_{12} = I_2$, each $\delta_*$ is a scalar unit point measure on
  $\Psi$, with $\delta_\infty$, $\delta_1$, $\delta_2$ and
  $\delta_{12}$ supported at $B$, $Bm_{p_1}$, $Bm_{p_2}$ and
  $Bm_{p_1}m_{p_2}$ (in case $Bm_{p_1}m_{p_2}\in \Psi$), respectively.
  If the algebra is $\mathcal A_B^0$ and $B$ has two or more distinct
  zeros, all with the same multiplicity, then $\mu = \delta_1 P_1 +
  \delta_2 P_2$.
\end{lemma}

For the algebra $\mathcal A_B$ with $N=2$, the $Bm_{p_1}m_{p_2}$ term
will not be present.

\begin{proof}
  Write $\mu_0$ for the restriction of $\mu$ in
  Lemma~\ref{lem:DJM-5.5} to $\Psi \backslash \{B\}$.  For $\gamma \in
  \mathbb C^2$, define the scalar valued Borel measure $\nu_\gamma :=
  \gamma^* \mu_0 \gamma$ on $\Psi \backslash \{B\}$.  Then by
  Lemmas~\ref{lem:DJM-5.5},~\ref{lem:DJM-4.3} and~\ref{lem:DJM-5.2}
  there are $2n^2$ functions $r_j:S \to \mathbb C^2$ with ranges
  contained in $E = \{k_{p_1} v_1, k_{p_2} v_2\}$ if the algebra is
  $\mathcal A_B$ ($n=1$), or $E = \{k_{p_1} v_1, {\tilde k}_1 v_1 ,
  \dots , {\tilde k}_{(n-1)N} v_1 \} \cup \{k_{p_2} v_2, {\tilde k}_1
  v_2 , \dots , {\tilde k}_{(n-1)N} v_2 \}$ in case of $\mathcal
  A_B^0$ ($n=N-1$), such that
  \begin{equation*}
    \gamma^*\left( \sum_{j=1}^{2n^2} r_j(z)r_j(w)^* \right) \gamma
    = \int_{\Psi \backslash \{B\}} K_\lambda(z)K_\lambda(w)^* \,
    d\nu_\gamma(\psi_\lambda).
  \end{equation*}

  Fix a set of $(n-1)^2+3$ non-zero points $X = \{z_j\} \subset S
  \backslash \zb$.  Using the notation introduced just before
  Lemma~\ref{lem:DJM-4.3}, set $K = \{{\tilde k}_j\}\cup
  \{k_{p_1},k_{p_2}\}$ for $\mathcal A_B^0$ and $K =
  \{k_{p_1},k_{p_2}\}$ for $\mathcal A_B$.  Define a codimension $1$
  subspace $V := \bigvee_{k\in K} (k(z))_{z\in X}$ in $\mathbb
  C^{(n-1)^2+3}$, and let $c = (\overline{c(z)})$ be a unit vector
  orthogonal to $V$.  Suppose that one of the entries of $c$ is zero.
  Take it to be $c(z_{(n-1)^2+3})$, reordering $X$ if necessary.  Then
  the vector $(\overline{c(z_j)})_1^{(n-1)^2+2} \in \mathbb
  C^{(n-1)^2+2}$ is orthogonal to $\bigvee_{k\in K}
  (k(z_j))_1^{(n-1)^2+2}$.  Since the latter spans $\mathbb
  C^{(n-1)^2+2}$, $c=0$, giving a contradiction.  So no $c(z_j)$ is
  $0$.

  Any $\gamma \in\mathbb C^2$ is in the span of the dual basis
  $\{w_1,w_2\}$ to $\{v_1,v_2\}$, and so for any $\gamma$,
  \begin{equation*}
    \sum_{z,w\in X} c(z)\gamma^*\left( \sum_{j=1}^{2n^2}
      r_j(z)r_j(w)^* \right) \gamma c(w)^* = 0.
  \end{equation*}
  Thus
  \begin{equation*}
    0 = \int_{\Psi \backslash \{B\}} {\left| \sum_{z\in X} 
        \psi_\lambda(z)c(z) \right|}^2 \, d\nu_\gamma(\psi_\lambda).
  \end{equation*}
  For $\mathcal A_B^0$ with $N > 2$, this has the form
  \begin{equation*}
    0 = \mu(\{zB\}) \|c\|^2 + \int_{\Psi \backslash \{B,zB\}} {\left|
        \sum_{z\in X} K_\lambda(z)c(z) \right|}^2 \,
    d\nu_\gamma(\psi_\lambda),
  \end{equation*}
  an immediate consequence of which is that $\mu(\{zB\}) = 0$.  Here,
  $K_\lambda = K_{B^{N-3}m_{\lambda}}$.  Otherwise, for $\mathcal
  A_B$,
  \begin{equation*}
    0 = \int_{\Psi \backslash \{B\}} {\left|
        \sum_{z\in X} K_\lambda(z)c(z) \right|}^2 \,
    d\nu_\gamma(\psi_\lambda),
  \end{equation*}
  where now if $\psi_\lambda = B G$, $G$ a Blaschke product,
  $K_\lambda = K_G$.

  For both algebras, it follows that
  \begin{equation*}
    \sum_{z\in X} K_\lambda(z)c(z) = 0 \qquad \nu_\gamma-a.e..
  \end{equation*}
  Let $\zb = \{\alpha_1,\dots,\alpha_N\}$.  Recall that for $\mathcal
  A_B$, $\lambda \in {\hat{\mathbb D}}^{N-1}$ and $K_\lambda$ is the
  vector $(\prod_{s=1}^{i-1} m_{\lambda_s} k_{\lambda_i})_{i=1}^N$,
  with the product absent when $i = 1$.  For $\mathcal A_B^0$,
  $\lambda \in \mathbb D$ and $K_\lambda = \begin{pmatrix}
    k_{\alpha_1} & \cdots & B^{N-2}k_\lambda \end{pmatrix}$.  In both
  cases, for $\nu_\gamma$ almost every $\psi_\lambda$, for every $i$
  and for $\lambda = \lambda_1$,
  \begin{equation*}
    \sum_{z\in X} B^{n-1}(z) k_\lambda \in V.
  \end{equation*}
  Thus for $z\in X$,
  \begin{equation*}
    B^{n-1}(z) k_\lambda (z) = \sum_{k\in K} c_k k(z).
  \end{equation*}
  Clearing the denominators, this gives for $i=0$,
  \begin{equation*}
    \begin{split}
      0 &= \prod (z-\alpha_\ell)^{n-1} (1-{\overline{p}}_1 z)
    (1-{\overline{p}}_2 z) \\ \quad &- (1-{\overline{\lambda}}_1 z)
    \left( a(z)(1-{\overline{p}}_1 z) (1-{\overline{p}}_2 z)  +
    (c_1(1-{\overline{p}}_1 z) + c_2 (1-{\overline{p}}_2 z)) \prod
    (1-{\overline{\alpha}}_\ell z)^{n-1} \right),
    \end{split}
  \end{equation*}
  where $a$ is a polynomial of degree at most $(n-1)^2$ and $c_1,c_2$
  are constants.  Therefore we have a polynomial of degree at most
  $(n-1)^2+2$ which is zero at $(n-1)^2+3$ distinct points, and so
  must be identically zero.  Thus the equation holds for all $z\in
  \mathbb C$.

  Rewrite this as
  \begin{equation*}
    \begin{split}
      & \left( \prod (z-\alpha_\ell)^{n-1} - a(z)
        (1-{\overline{\lambda}}_1 z) \right) (1-{\overline{p}}_1 z)
      (1-{\overline{p}}_2 z) \\
      = & \prod (1-{\overline{\alpha}}_\ell z)^{n-1}
      (c_1(1-{\overline{p}}_1 z) + c_2 (1-{\overline{p}}_2 z))
      (1-{\overline{\lambda}}_1 z).
    \end{split}
  \end{equation*}
  The left side has zeros at $1/{\overline{p}}_1$ and
  $1/{\overline{p}}_2$, so the right must as well, implying that for
  a constant $\tilde c$,
  \begin{equation*}
    (c_1(1-{\overline{p}}_1 z) + c_2 (1-{\overline{p}}_2 z))
    (1-{\overline{\lambda}}_1 z) = \tilde c (1-{\overline{p}}_1 z)
    (1-{\overline{p}}_2 z).
  \end{equation*}
  If $\tilde c = 0$, then $c_1 = c_2 = 0$.  This forces $\prod
  (z-\alpha_\ell)^{n-1} = a(z) (1-{\overline{\lambda}}_1 z)$, which is
  impossible.  So $\tilde c \neq 0$, and $\lambda = p_1$ and $c_2 =
  0$, or $\lambda = p_2$ and $c_1 = 0$.  In the case of $\mathcal
  A_B^0$, there is nothing further to prove.

  For $\mathcal A_B$, assume $\lambda_1 = p_1$, and suppose that
  \begin{equation*}
    m_{p_1}k_{\lambda_2} = c_1 k_{p_1} + c_2 k_{p_2}.
  \end{equation*}
  Clearing the denominators gives a quadratic polynomial which is zero
  at the three points of $X$, and hence in all of $\mathbb C$.
  Therefore this equation holds meromorphically in $\hat{\mathbb C}$.
  If $c_2 = 0$, then $\lambda_2 = \infty$, which is the case already
  considered.  Likewise it may be assumed that $c_1 \neq 0$.  Thus
  both sides have poles at $1/{\overline{p}}_1$ and
  $1/{\overline{p}}_2$; that is, $\lambda_2 = p_2$.  The case where
  $\lambda_1 = p_2$ is identical, and then $\lambda_2 = \infty$ or
  $p_1$.

  Next assume that $B$ has $3$ or more zeros and
  \begin{equation*}
    m_{p_1}m_{p_2}k_{\lambda_3} = c_1 k_{p_1} + c_2 k_{p_2}.
  \end{equation*}
  Clearing denominators again gives a quadratic polynomial which is
  zero at the three points of $X$, and hence in all of $\mathbb C$.
  The equation thus holds meromorphically in $\hat{\mathbb C}$.
  Examining the poles, it is clear that the only possibility is for
  $\lambda_3 = \infty$, and so this reduces to the last case
  considered.

  Consequently, for $\mathcal A_B$ there are at most three distinct
  possibilities, where exactly one $\lambda_j = p_1$ and the rest are
  $\infty$, where exactly one $\lambda_j = p_2$ and the rest are
  $\infty$, and where one $\lambda_j = p_1$ another equals $p_2$ and
  the rest are $\infty$.
\end{proof}

\begin{theorem}
  \label{thm:rational-dilation-fails}
  Suppose that $\Phi = B(z)R(z) \in M_2\otimes \mathcal A_B$ and $\Phi
  = B(z)^{N-1}R(z) \in M_2\otimes \mathcal A^0_B$, with $R$ chosen as
  in \eqref{eq:26}, and $B$ having two or more distinct zeros all with
  the same multiplicity in the latter case.  Then
  $\Delta_{\Phi,S}\notin C_{2,S}$.  Consequently, there are
  contractive unital representations of $\mathcal A_B$ and $\mathcal
  A_B^0$ which are not $2$-contractive, and hence not completely
  contractive.
\end{theorem}

\begin{proof}
  From Lemma~\ref{lem:DJM-5.6}, for $\mathcal A_B$,
  \begin{equation*}
    R(z)R(w)^* = m_{p_1}(z)m_{p_1}(w)^* P_1 +
    m_{p_2}(z)m_{p_2}(w)^* P_2 +
    m_{p_1}(z)m_{p_2}(z)m_{p_2}(w)^*m_{p_1}(w)^* P_{12} +
    P_{\infty}
  \end{equation*}
  Hence
  \begin{equation*}
    R(p_1)R(p_2)^* = \frac{1}{2}
      \begin{pmatrix}
        0 & 0 \\ m_{p_1}(p_2)^* & 1
      \end{pmatrix}
      = P_\infty \geq 0.
  \end{equation*}
  Thus $p_2 = p_1$, a contradiction.

  For $\mathcal A_B^0$, by Lemma~\ref{lem:DJM-5.6},
  $B^{N-1}(z)R(z)R(w)^*B^{N-1}(w)^* = m_{p_1}(z)m_{p_1}(w)^* P_1 +
  m_{p_2}(z)m_{p_2}(w)^* P_2$.  This time, evaluating at $z= p_1$,
  $w=p_2$ gives
  \begin{equation*}
    0 = \frac{1}{2} B^{N-2}(p_1)B^{N-2}(p_2)^* 
    \begin{pmatrix}
      0 & 0 \\ m_{p_1}(p_2)^* & 1
    \end{pmatrix},
  \end{equation*}
  which is only possible if $p_1$ or $p_2 \in \zb$, contrary to
  assumption.

  Thus $\Delta_{\Phi,S}\notin C_{2,S}$.
\end{proof}

\begin{corollary}
  \label{cor:no-rat-dil}
  For any finite Blaschke product $B$ with two or more zeros, rational
  dilation fails over the distinguished variety $\mathcal V_B
  \subseteq {\overline{\mathbb D}}^N$.  If $B$ has two or more
  distinct zeros all of the same multiplicity, rational dilation fails
  over the distinguished variety $\mathcal N_B \subseteq
  {\overline{\mathbb D}}^2$.
\end{corollary}

\section{Conclusion}
\label{sec:conclusion}

As mentioned at the end of Section~\ref{sec:dist-vari-assoc}, for
$2\leq n \leq N$, and a Blaschke product $B$ with $N$ zeros, one can
consider the algebra generated by $B, Bz, \dots, Bz^n$.  This will be
completely isometrically isomorphic to an algebra of holomorphic
functions on a distinguished variety in ${\overline{\mathbb D}}^n$,
and arguments similar to those given here can be used to study such
algebras and the rational dilation problem over the attendant
varieties.

We speculate that rational dilation fails on $\mathcal A(\mathcal
N_B)$ even without the restrictions imposed here.  There are a few
cases we have been able to verify (for example, $B$ with three zeros,
two of which are the same).  However, what on the surface should be
the easiest case ($B$ has three or more zeros which are all the same)
resists our approach, at least with the $\Phi$ used here.

For both $\mathcal A(\mathcal N_B)$ and $\mathcal A(\mathcal V_B)$ (as
well as the other algebras mentioned), there is thus a hierarchy of
unital representations.  The nicest, though smallest class, are those
which are completely contractive.  Next are the contractive
representations, a class that in a small coterie of examples (eg,
simply connected planar sets with smooth boundaries, the bidisk,
annuli, the symmetrized bidisk~\cite{MR3641771}) agrees with the
completely contractive representations.  In the context of this paper,
it is however strictly larger, since there are examples of contractive
representations which are not $2$-contractive, much as with the
tri-disk~\cite{MR0268710}.  It is natural to wonder if there are
$2$-contractive representations which are not $3$-contractive, and so
on.  Possibly some of the ideas presented here, along with the work of
Ball and Guerra Huam\'an~\cite{MR3138369,MR3057412}, could enable the
construction of a minimal set of test functions for $M_2(\mathbb
C)\otimes\mathcal A(\mathcal N_B)$ and $M_2(\mathbb C)\otimes\mathcal
A(\mathcal V_B)$, which would be a first step in analyzing this
question.

Function algebras on distinguished varieties are intimately connected
to function algebras on multiply connected
domains~\cite{MR0241629,MR2710343} (see also~\cite{MR3584680} and
\cite{MR2661491}).  Perhaps the techniques employed here will enable
the extension of the results in~\cite{MR2643788} to general multiply
connected domains.

Finally, there is the class of bounded unital representations which
send the algebra generators to contractions.  As in the case of the
tridisk~\cite{MR0355642}, these are seen to comprise a strictly larger
class of representations than that of the contractive representations
of $\mathcal A(\mathcal V)$ for the varieties $\mathcal V$ we
considered here.  By comparison, for the disk, bidisk and symmetrized
bidisk, such representations are automatically contractive, and so as
noted, completely contractive.  This is not universal though.
Consider an annulus $\mathbb A$ (assumed without loss of generality to
be centered at $0$ with outer radius $1$ and inner radius $r$).  Any
minimal set of test functions over this set is infinite, so there are
representations sending the generators $z$ and $r/z$ of $A(\mathbb A)$
to contractions which are not contractive representations.  On the
other hand, there is a uniform bound for the norm growth in this case,
since the von~Neumann inequality holds up to multiplication of the
function norm by $K> 1$ (see~\cite{MR862189,MR2449098}, as well as
\cite{MR3653945} and \cite{MR3656275}).  This relation between
spectral and complete $K$-spectral sets is another area worth
exploring in $\mathcal A(\mathcal V)$.

\end{document}